\documentclass[article,lineno]{biometrika}

\usepackage{amsmath}

\usepackage{times}
\usepackage{bm}
\usepackage{natbib}

\usepackage[plain,noend]{algorithm2e}

\usepackage{caption}

\nolinenumbers 

\makeatletter
\renewcommand{\algocf@captiontext}[2]{#1\algocf@typo. \AlCapFnt{}#2} 
\def\@algocf@capt@plain{top}
\renewcommand{\algocf@makecaption}[2]{%
  \addtolength{\hsize}{\algomargin}%
  \sbox\@tempboxa{\algocf@captiontext{#1}{#2}}%
  \ifdim\wd\@tempboxa >\hsize
    \hskip .5\algomargin%
    
    \parbox[t]{\hsize}{\algocf@captiontext{#1}{#2}}
  \else%
    \global\@minipagefalse%
    \hbox to\hsize{\box\@tempboxa}
  \fi%
  \addtolength{\hsize}{-\algomargin}%
}
\makeatother

\def\T{{ \mathrm{\scriptscriptstyle T} }}

\begin{document}




\markboth{Yang Wang \and Zhangsheng Yu Ph.D}{Time-varying effects for panel count data}

\title{A kernel regression model for panel count data with time-varying coefficients}

\author{Yang Wang}
\affil{Department of Statistics, School of Mathematical Sciences and SJTU-Yale Joint Center for Biostatistics, Shanghai Jiao Tong University, Shanghai, China.
\email{wy910028@sjtu.edu.cn} }
\author{Zhangsheng Yu Ph.D}
\affil{Department of Statistics, School of Mathematical Sciences and SJTU-Yale Joint Center for Biostatistics, Shanghai Jiao Tong University, Shanghai, China.
\email{yuzhangsheng@sjtu.edu.cn}}

\maketitle

\begin{abstract}
For the conditional mean function of panel count model with time-varying coefficients, we propose to use local kernel regression method for estimation. Partial log-likelihood with local polynomial is formed for estimation. Under some regularity conditions, strong uniform consistency rates are obtained for the local estimator. At target time point, we show that the local estimator converges in distribution to normal distribution. The baseline mean function estimator is also shown to be consistent. Simulation studies show that the time-varying coefficient estimator is close to the true value, the empirical coverage probabilities of the confidence interval is close to the nominal level. We also applied the proposed method to analyze a clinical study on childhood wheezing.
\end{abstract}
 

\section{Introduction}

Panel count data occurred when events are observed at finite fixed time points and the visit times vary from subject to subject, and the exact event times between two consecutive observation times are unknown. In reality, panel count data are often encountered in clinical, demographical and industrial researches. For example, in an observational study on childhood asthma (\citet{Tepper:2008}), the number of wheezing episodes experienced by each child between two consecutive interviews were collected by phone call. The event number may be greater than one, and the exact time of wheezing occurrence was unknown. The wheezing event time analysis should be considered as the panel count data type. Meanwhile, the risk factors' effect on the panel count outcome may vary over time. Therefore, it is desirable to study the panel count model with nonparametric time-varying coefficients.

In the past three decades, there have been extensive researches to study the proportional mean model for panel count data. Generally, there are two main approaches, one is pseudo-likelihood estimation method, and the other is the estimating equation approach. For the likelihood method, pseudo-likelihood function was constructed based on the nonhomogeneous Poisson process assumptions, see \citet{Zhang:2002}, \citet{WellZh:2007}. \citet{Zhu:2018} developed a likelihood-based semi-parametric regression model for panel count data under the same assumptions. \citet{HuaZh:2014} proposed sieve maximum likelihood method under the Gamma-Frailty inhomogeneous Poisson process assumption. For the estimating equation approach, \citet{HuSun:2003}, \citet{Sun:2007} and \citet{Li:2010} discussed estimating equation approach to analyze the semi-parametric regression model for panel count data with correlated observation times. \citet{He:2007}, \citet{Li:2011} and \citet{Li:2015} proposed estimating equation approach for regression analysis of multivariate panel count data. All above methods focused on the parametric covariate effect estimation which lead to bias estimators when the covariate effect changes over time. Therefore, statistical methods dealing with time-varying coefficient for panel count data are much desired.

In this paper, we focus on the nonparametric time-varying coefficient estimation for panel count data. For nonparametric regression model, there are two main approaches, kernel estimation and spline method, generally used to study for survival data in recent researches. For example, \citet{CaiSun:2003}, \citet{Tian:2005}, \citet{CaiFan:2007}, \citet{YuLin:2010} and \citet{Lin:2016} discussed kernel-weighted likelihood method for Cox model with time-varying effects. \citet{BuSa:2011}, \citet{Aris:2013} and \citet{Aris:2014} proposed B-spline methods for time-varying effects model in survival data analysis. Nevertheless, limited work has been done in nonparametric panel count model. \citet{ZTY:2018} investigated B-splined pseudo-likelihood method for time-varying coefficients model of panel count data. However, they only provided the asymptotic distribution for an integral of the estimator instead of the proposed estimator. Furthermore, the robustness of splined estimation was poor depending on the knot numbers. To make up for these deficiencies, we employ the kernel estimated approach to study time-varying coefficients panel count model. Under the assumption of nonhomogeneous Poisson process, we construct kernel-weighted local partial log-likelihood for estimation and provide the asymptotic properties of the coefficient estimators.

In the following, $\S$\,2 presents time-varying coefficients mean model, and the kernel-weighted local partial log-likelihood for estimation. $\S$\,3 derives the asymptotic theoretical properties of estimators based on modern empirical process theories. $\S$\,4 describes the numerical results obtained from simulation studies to exam the proposed model. $\S$\,5 applies the proposed approach to a child wheeze study. $\S$\,6 concludes the paper with a discussion. The technical details are presented in Appendices.

\section{The mean function model and local partial log-likelihood estimation}

\subsection{The conditional mean function model}

We first introduce some notations. Let $\{N_i(t),t\geq 0\}$ be a counting process of the cumulative number of events up to time $t, 0 \leq t \leq \tau$, and $\tau$ is the maximum follow up time. Without loss of generality, we assume that $N_i(0)=0,i=1,2,\ldots, n$. For subject $i$, the patient is followed at time $\{T_{il}:0<T_{i1}<T_{i2}<\cdots<T_{i k_i}<\infty \}$, where $k_i$ and $T_{il}$ are random. We denote $\{O_i(t),t\geq 0\}$ as the observation process, which is a point process $O_{i}(t)=\sum_{l=1}^{k_{i}}I(T_{il}\leq t)$, $t\geq 0$, representing the cumulative visit numbers up to time $t$. Here, $I(\cdot)$ is the indicator function. Let $o_{i}(t)=O_{i}(t)-O_i(t-)$, so that $o_{i}(t)$ denotes whether subject $i$ has a visit at time $t$. Suppose that $C_i, i=1,2,\ldots, n$, are censoring times. And $N_i(T_{il})$ is not observed when $C_i<T_{il}<\tau$. Also let $\{ Z_i, i=1,\ldots, n\}$ be $d$-dimensional covariates. In this paper, for simplicity, we consider $d=1$. Suppose that given $Z_i$, the mean function of $N_i(t)$ is
\begin{equation}
\label{MF1}
E\{N_i(t)\mid Z_i=z_i\}=\mu_0(t)\exp(\beta(t)z_i),\quad t\geq 0,
\end{equation}
where the baseline function $\mu_0(t)$ is unspecified, and $\beta(t)$ is an unknown function. In this study, we assume that $\{N_i(t),O_i(t),C_i,Z_i\},i=1,\ldots, n$, are independent and identically distribution. Furthermore, we assume that $N_i(t),O_i(t)$ and $C_i$ are independent, given the covariate $Z_i$. 

\subsection{Local kernel estimation conditional on observation process}

As the information about recurrent process $N_i(t)$ can be observed at visit times, we define a new counting process $\widetilde{N}_i(t)$, with respect to subject $i$, conditional on observation process:
\begin{equation}
\label{M1}
\widetilde{N}_i(t)=\int_{0}^{t}N_i(u)dO_i(u),\quad t\geq 0.
\end{equation}
The defined process only jumps at the observation times $\{T_{i,l}, l=1,\ldots, k_i\}$, and the jump size is $N_i(T_{i,l})$. Then, conditional on  the observation process $O_i(t)$ and covariate  $Z_i$, the mean of $d\widetilde{N}_i(t)$ is as follows:
\begin{equation}
\label{MF2}
E\{d\widetilde{N}_i(t)\mid Z_i=z_i;O_i(u),0<u\leq t\}=\mu_0(t)\exp(\beta(t)z_i)dO_i(t).
\end{equation}

Suppose that $d\widetilde{N}_i(t)$ is nonhomogeneous Poisson process, we can construct the logarithm of partial likelihood function with observed information over $[0,\tau]$ $(\tau>0)$ by employing similar techniques presented in work of \citet{LawNad:1995} and \citet{HuSun:2003}, as follows:
\begin{equation}
\label{L1}
pl_{n}(\beta(u))=n^{-1}\sum_{i=1}^{n}\int_{0}^{\tau}I(C_i\geq u)\big \{\beta(u)z_i-\log n^{-1}\sum_{j=1}^{n}I(C_j\geq u)\exp(\beta(u)z_j)o_j(u)\big \}d\widetilde{N}_i(u).
\end{equation}

To estimate the time-varying coefficient, we employ the kernel likelihood approach. For each fixed time point $t$, using Taylor expansion, we approximate the $\beta(u)$ with the $p$th-order polynomial as:
\begin{equation}
\label{a1}
\beta(u)\approx \beta(t)+\beta'(t)(u-t)+\cdots+\beta^{(p)}(t)(u-t)^{p}/p!,
\end{equation}
Set $\bm{\beta}=(\beta_{0}(t),\beta_{1}(t),\ldots,\beta_{p}(t))^{\T}=(\beta(t),\beta'(t),\ldots,\beta^{(p)}(t)/p!)^{\T}$ and $z_i(\bm{u})=z_i(1,u-t,\ldots,(u-t)^{p})^{\T}$. Let $K(\cdot)$ be a kernel function which can down weight the likelihood contribution of remote time points, and let $h$ be the bandwidth that can regulate the local neighborhood sizes. Then, by inserting localizing weights, with the local polynomial equation \eqref{a1}, we obtain the local partial log-likelihood:
\begin{align}
\label{L2}
&\mathcal{L}_{n}(\bm{\beta})=\nonumber\\
&n^{-1}\sum_{i=1}^{n}\int_{0}^{\tau}K_h(u-t)I(C_i\geq u)\big \{\bm{\beta}^{\T}z_i(\bm{u})-\log n^{-1}\sum_{j=1}^{n}I(C_j\geq u)\exp(\bm{\beta}^{\T}z_j(\bm{u}))o_j(u)\big \}d\widetilde{N}_i(u),
\end{align}
where $K_h(\cdot)=h^{-1}K(\cdot/h)$. 

Let $\bm{\widehat{\beta}}$ be the maximizer of \eqref{L2} with respect to $\bm{\beta}$. Then $\widehat{\beta}(t)=\widehat{\beta}_{0}(t)$ is the local kernel partial maximum likelihood estimator of $\beta(t)$, which is the first component of vector $\bm{\widehat{\beta}}$. 

To obtain the maximizer of \eqref{L2}, we introduce some additional notations. Let 
\begin{equation}
\label{S1}
\widetilde{S}_{n,j}(u,\bm{\beta})=n^{-1}\sum_{i=1}^{n}I(C_i\geq u)\exp(\bm{\beta}^{\T}z_i(\bm{u}))o_i(u)z_i(\bm{u})^{\otimes j},\quad j=0,1,2.
\end{equation}\\
Then, \eqref{L2} can be modified as follows:
\begin{equation}
\label{L3}
\mathcal{L}_{n}(\bm{\beta})=n^{-1}\sum_{i=1}^{n}\int_{0}^{\tau}K_h(u-t)I(C_i\geq u)\big \{\bm{\beta}^{\T}z_i(\bm{u})-\log \widetilde{S}_{n,0}(u,\bm{\beta})\big \}d\widetilde{N}_i(u).
\end{equation}\\
And we can derive the local kernel estimating equation,
\begin{equation}
\label{L4}
\mathcal{L}'_{n}(\bm{\beta})=n^{-1}\sum_{i=1}^{n}\int_{0}^{\tau}K_h(u-t)I(C_i\geq u)\{z_i(\bm{u})-\widetilde{S}_{n,1}(u,\bm{\beta})/\widetilde{S}_{n,0}(u,\bm{\beta})\}d\widetilde{N}_{i}(u).
\end{equation}
which is the gradient of $\mathcal{L}_{n}(\bm{\beta})$. Again, the Hessian matrix of $\mathcal{L}_{n}(\bm{\beta})$ is formed as 
\begin{align}
\label{L5}
&\mathcal{L}''_{n}(\bm{\beta})=\nonumber\\
&-n^{-1}\sum_{i=1}^{n}\int_{0}^{\tau}K_h(u-t)I(C_i\geq u)\{\widetilde{S}_{n,2}(u,\bm{\beta})/\widetilde{S}_{n,0}(u,\bm{\beta})-(\widetilde{S}_{n,1}(u,\bm{\beta})/\widetilde{S}_{n,0}(u,\bm{\beta}))^{\otimes 2}\}d\widetilde{N}_i(u).
\end{align}
By Cauchy-Schwarz inequality, we can check that the right-hand side of \eqref{L5} is negative, as $n\to \infty$. Thus, $\mathcal{L}_{n}(\bm{\beta})$ is strictly concave with respect to $\bm{\beta}$. Hence, there is a unique maximizer of the local likelihood $\mathcal{L}_{n}(\bm{\beta})$. Then, using the Newton-Raphson algorithm, we can get the local kernel estimator $\bm{\widehat{\beta}}$.
Here, the $(j+1)$th step of Newton-Raphson algorithm is 
\[
\bm{\widehat{\beta}}^{(j+1)}=\bm{\widehat{\beta}}^{(j)}-\mathcal{L}'_{n}(\bm{\widehat{\beta}}^{(j)})/\mathcal{L}''_{n}(\bm{\widehat{\beta}}^{(j)}),
\]
where $\bm{\widehat{\beta}}^{(j)}$ is the value at $j$th iteration.

After obtaining the $\widehat{\beta}(t)=\widehat{\beta}_{0}(t)$ at each observation time, we can construct the Breslow type estimator $\widehat{\mu}_{0}(t)$ for the baseline mean function $\mu_{0}(t)$ as $\widehat{\mu}_0(t)=\sum_{i=1}^{n}I(C_i\geq t)N_i(t)o_i(t)/\sum_{i=1}^{n}I(C_i\geq t)\exp(\beta(t)z_i)o_i(t)$ (\citet{Breslow:1974} and \citet{Cox:1992}). Substituting $\beta(t)$ by $\widehat{\beta}(t)$, we obtain the baseline estimator 
\begin{equation}
\label{mm}
\widehat{\mu}_0(t,\widehat{\beta}(t))=\sum_{i=1}^{n}I(C_i\geq t)N_i(t)o_i(t)/\sum_{i=1}^{n}I(C_i\geq t)\exp(\widehat{\beta}(t)z_i)o_i(t).
\end{equation}

\section{Asymptotic properties}

\subsection{Strong uniform consistency and asymptotic normality}

In this section, we present the asymptotic theoretical properties of the proposed estimator. For simplicity of presentation, we introduce some notations. Let
$\bm{u}=(1,u,\ldots,u^{p})^{\T}$, $\Omega_{1}=\int K(u) \bm{u}^{\otimes 2}du$, $\Omega_{2}=\int K^{2}(u) \bm{u}^{\otimes 2}du$. Set $H= diag(1,h,\ldots,h^{p})$, $ \bm{u}-\bm{t}=(1,(u-t)/h,\ldots, (u-t)^{p}/h^{p})^{\T}$, and the true value $\bm{\beta}^{\ast}=(\beta(t),\beta'(t)\ldots,\beta^{(p)}(t)/p!)^{\T}$.

$p_1(t \mid z) =pr(C \geq t \mid Z=z), \quad p_2(t\mid z) =pr(o(t)\mid Z=z), \quad \mu(t \mid z) =\mu_0(t)\exp(\beta(t)z)$,

$\sigma(t \mid z) =\mu^{2}_{0}(t)\exp(2\beta(t)z),\quad q_{j}(t)=E(p_1(t \mid z)p_2(t \mid z)\mu(t \mid z)z^{j}), \quad j=0,1,2$.\\
Define
\begin{equation}
\sigma_1(t)=q_{2}(t)-q_1^{2}(t)/q_0(t),\quad
\sigma_2(t)=E(p_1(t \mid z)p_2(t \mid z)(z-q_{1}(t)/q_{0}(t))^{2}\sigma(t \mid z)).
\end{equation}
Let $T=\{t: t\in [0,\tau]\}$.

The following regularity conditions are required for the theorems and lemmas.
\begin{condition}
The kernel function $K(\cdot)\geq 0$ is a symmetric density function with compact support $[-1,1]$, and is bounded variation taking the value as zero at the boundaries;
\end{condition}
\begin{condition}
$N(\cdot)$, $O(\cdot)$ are bounded, $E(N^{2}(\cdot)\mid Z=z)$ is exist and $E(Z^{\lambda})^{1/\lambda}<\infty$, for $2<\lambda<\infty$;
\end{condition}
\begin{condition}
The time-varying coefficient $\beta(t)$ is $(p+1)$th-order continuous differentiable with bounded variation in $T$;
\end{condition}
\begin{condition}
$\mu_0(t)$, $p_1(t \mid z)$, $p_2(t \mid z)$ and $\beta(t)z$ are positive and continuous in $T$;
\end{condition}
\begin{condition}
$q_{0}(t)>0$, $q_{1}(t)$, $q_{2}(t)$, $\sigma_1(t)$ and $\sigma_2(t)$ are continuous, and $\inf \sigma_{1}(t)=M_{1}<\infty$, $\sup q_{1}(t)/q_{0}(t)=M_{2}<1$, $\sup q_{0}(t)=M_{3}<\infty$.
\end{condition}

The above conditions will be used to prove the strong uniform consistency and pointwise asymptotic normality of the proposed estimator. Conditions 1-3 are technical and regularity conditions. Conditions 4 and 5 are necessary for deriving the uniform convergence result. Among them, we assume $p_{1}(t \mid z)> 0$ and $p_{2}(t \mid z)> 0$, which ensure that there is at least one event on each $t\in T$ as $n$ gets large enough. This is crucial to theoretical demonstration of asymptotic properties. Next up, we state the main results of this paper. The detailed proofs are relegated to Appendices.

\begin{theorem}
\label{theorem1}
Under Conditions 1-5, assume that the bandwidth $h$ satisfies the conditions:
\[ h\to 0,\quad nh/\log n\to \infty \quad and \quad h\geq (\log n/n)^{1-2/\lambda} \quad for \quad \lambda> 2,\] 
then there exists a sequence of solutions $\big \{\bm{\widehat{\beta}}=(\widehat{\beta}_{0}(t),\ldots,\widehat{\beta}_{p}(t))^{\T}\big \}$ to equation \eqref{L4}, such that, for each $k=0,\ldots,p$, almost surely
\begin{equation}
\label{them1}
\sup_{t\in T}|\widehat{\beta}_{k}(t)-\beta^{(k)}(t)/k!|=O(h^{-k}\{(\log n/(nh))^{1/2}+h\}) \quad as \quad  n\to \infty.
\end{equation}
Especially, when the local linear approximation is used ($p=1$), we have, almost surely
\begin{equation}
\label{them1.1}
\sup_{t\in T}|\widehat{\beta}(t) - \beta(t)|=O((\log n/(nh))^{1/2}+h) \quad as \quad  n\to \infty.
\end{equation}
\end{theorem}

The above theorem shows that the proposed estimator is strong uniformly consistent. This indicates the local estimator is uniform asymptotically unbiased as $n\to \infty$. Under more stringent conditions, the strong uniform consistency rate of the proposed estimator is similar to that of \citet{Zhao:1994} and \citet{Claeskens:2003}. In their paper, they discussed the strong uniform convergence rate for the nonparametric location regression problem. Here, the strong uniform consistency of the proposed estimator is derived based on Lemma~\ref{lemB1} and Lemma~\ref{lemB2} presented in Appendix 2. In particular, Lemma~\ref{lemB1} discusses the supremum of the local kernel estimating equation \eqref{L4} under some conditions, which play a crucial role in the proof of Theorem~\ref{theorem1}. The detailed proofs are presented in Appendix 2.

\begin{theorem}
\label{theorem2}
Under Conditions 1-5, assume that the bandwidth $h$ satisfies the conditions:
\[h\to 0, \quad nh\to \infty, \quad and \quad nh^{2p+3} \quad is \quad bounded,\]
then the asymptotic distribution of $\bm{\widehat{\beta}}$ satisfies
\begin{equation}
\label{them2}
(nh)^{1/2}\big\{H(\bm{\widehat{\beta}}-\bm{\beta}^{\ast})-\Omega_{1}^{-1}\bm{b}h^{p+1}\beta^{(p+1)}(t)/(p+1)!\big\}\rightarrow N(0,\sigma_1^{-2}(t)\sigma_2(t)\Omega_{1}^{-1}\Omega_{2}\Omega_{1}^{-1}), 
\end{equation}\\
where $\bm{b}=\int u^{p+1}\bm{u}K(u)du$.
\end{theorem}

The result in above theorem demonstrates the asymptotic normality of the proposed estimator, under general conditions. The $\bm{\widehat{\beta}}$ converges in the optimal rate of kernel estimators and analogous to the spline estimator. The bias is of order $h^{p+1}$ and related to the $(p+1)$-derivative of real function $\beta(t)$. Hence, it tends to zero when the bandwidth gets to zero. The theorem also gives the joint asymptotic normality of the estimator for derivatives. Particularly, the variance and bias of $\widehat{\beta}^{(r)}(t)=\widehat{\beta}_{r}(t)$ can be obtained by the $r$th component of \eqref{them2}. The detailed proof is presented in Appendix 2. We also present the two lemmas which are key to the proof of Theorem~\ref{theorem1} and Theorem~\ref{theorem2} in Appendix 1. When the local linear approximation is used ($p=1$), we have the following corollary:

\begin{corollary}
\label{coro1}
Under Conditions 1-5, and assume that the bandwidth $h$ satisfies the conditions: 
\[h\to 0, \quad nh\to \infty, \quad and \quad nh^{5} \quad is \quad bounded,\]
then the asymptotic distribution of $\widehat{\beta}(t)$ satisfies
\begin{equation}
\label{cor1}
(nh)^{1/2}\big \{\widehat{\beta}(t)-\beta(t)-2^{-1}\mu_2h^{2}\beta''(t)\big\}\to N(0,\nu_0\sigma_1^{-2}(t)\sigma_2(t)), 
\end{equation}\\
where $\mu_2=\int u^{2}K(u)du$, $\nu_0=\int K^{2}(u)du.$
\end{corollary}

The estimator of nonparametric $\beta(t)$ is asymptotically normal. The bias is of order $h^{2}$ and related to the second derivative of time-varying function $\beta(t)$. As consequence of  \eqref{cor1}, by minimizing the weighted mean integrated squared error:
\begin{equation}
\label{error}
\int_{0}^{\tau} \{4^{-1}\mu_{2}^{2}h^{4}\beta''^{2}(t)+\nu_{0}\sigma_{1}^{-2}(t)\sigma_{2}(t)/(nh)\}w(t)dt,
\end{equation}
we can derive the theoretical optimal bandwidth for $\widehat{\beta}(t)$, as follows:
\begin{equation}
\label{hopt}
h_{opt}=\{\nu_{0}\int_{0}^{\tau}\sigma_{1}^{-2}(t)\sigma_{2}(t)w(t)dt/(\mu_{2}^{2}\int_{0}^{\tau}\beta''^{2}(t)w(t)dt)\}^{1/5}n^{-1/5}.
\end{equation}.

\subsection{Estimation of covariance matrix}

We propose the covariance estimator of $\bm{\widehat{\beta}}$ based on its asymptotic covariance by plugging the estimated $\bm{\beta}$ into covariance in \eqref{them2}, as follows:
\begin{equation}
\label{Sig}
 \widehat{\Sigma}(t)=\widehat{\Sigma}_1^{-1}(t)\widehat{\Sigma}_2(t)\widehat{\Sigma}_1^{-1}(t),
 \end{equation}
 where
\begin{align}
\label{Sig1}&\widehat{\Sigma}_1(t)=n^{-1}\sum_{i=1}^{n}\int_0^{\tau}K_h(u-t)(\bm{u}-\bm{t})^{\otimes 2}I(C_i\geq u)V_{1}(u,\bm{\widehat{\beta}})d\widetilde{N}_i(u),\\
\label{Sig2}&\widehat{\Sigma}_2(t)=n^{-1}\sum_{i=1}^{n} \int_0^{\tau}hK_h^{2}(u-t)(\bm{u}-\bm{t})^{\otimes 2}I(C_i\geq u)V_{2}(u,\bm{\widehat{\beta}})\widehat{\mu}_{0}^{2}(u,\widehat{\beta}(u))\exp(2\bm{\widehat{\beta}}^{\T}z_i(\bm{u}))o_i(u)du,
\end{align}
with 
\begin{align}
\label{V1}&V_{1}(u,\bm{\widehat{\beta}})=S_{n,2}(u,\bm{\widehat{\beta}})/S_{n,0}(u,\bm{\widehat{\beta}})-\{S_{n,1}(u,\bm{\widehat{\beta}})/S_{n,0}(u,\bm{\widehat{\beta}})\}^{2},\\
\label{V2}&V_{2}(u,\bm{\widehat{\beta}})=\big \{z_i-S_{n,1}(u,\bm{\widehat{\beta}})/S_{n,0}(u,\bm{\widehat{\beta}})\big \}^{ 2},\\
\label{Sj}&S_{n, j}(u,\bm{\widehat{\beta}})=n^{-1}\sum_{i=1}^{n}I(C_{i}\geq u)\exp(\bm{\widehat{\beta}}^{\T}z_{i}(\bm{u}))o_{i}(u)z_{i}^{j}, \quad j = 0, 1, 2.
\end{align}

We show that $\widehat{\Sigma}_1(t)$ and $\widehat{\Sigma}_2(t)$ converges in probability to $\Sigma_{1}(t)$ and $\Sigma_{2}(t)$ presented in Appendix 2, respectively. Therefore the estimator $\widehat{\Sigma}(t)$ of the asymptotic covariance $\Sigma(t)=\sigma_1^{-2}(t)\sigma_2(t)\Omega_{1}^{-1}\Omega_{2}\Omega_{1}^{-1}$ in \eqref{them2} is consistent. And detailed proofs are displayed in Appendix 2. Moreover, the finite sample performance of the variance estimation is validated in simulation studies.

\subsection{Asymptotic properties of baseline mean function}

As introduced in $\S$\,2$\cdot$2, we use Breslow type estimator to evaluate the baseline mean function at each fixed time point. Here, we discuss the asymptotic properties of the estimator $\widehat{\mu}_0(t,\widehat{\beta}(t))$.
\begin{theorem}
\label{theorem3}
Under Conditions 1-5, assume that the bandwidth $h$ satisfies the conditions:
\[h\to 0, \quad nh\to \infty, \quad and \quad nh^{5}=o(1),\]
 then the asymptotic distribution of $\widehat{\mu}_0(t,\widehat{\beta}(t))$ satisfies
\begin{equation}
\label{them3}
(nh)^{1/2}(\widehat{\mu}_0(t,\widehat{\beta}(t))-\mu_0(t))\to N(0,\Sigma_3(t)),
\end{equation}
where $\Sigma_3(t)=\nu_0q_0^{-2}(t)q_1^{2}(t)\sigma_1^{-2}(t)\sigma_2(t)$, and $\nu_0=\int K^{2}(u)du$.
\end{theorem}

The detailed proofs are presented in Appendix 2. Furthermore, the rate of convergence for $\widehat{\mu}_{0}(t,\widehat{\beta}(t))$ is $(nh)^{1/2}$ which is the same as the rate of $\widehat{\beta}(t)$. And the finite sample performance of the estimator is displayed in simulation studies.

\section{Simulation}

In this section, we evaluated the finite sample performance of the proposed local kernel estimator through a numerical study. In each simulated data set, we generated $n$ independent and identically distributed random variables $\{K_i,T_i, N_i, Z_i\}$. For each individual $i$, the number of observation $K_i$ was generated as a discrete uniform distribution on $\{1, 2, \ldots, C\}$, where the number $C$ was finite. And the follow-up time $T_i=(T_{i1}, \ldots, T_{iK_i})$ were generated as an exponential distribution. The covariate $Z_i$ was generated from uniform distribution $U(0,1)$. Given the time-varying coefficient $\beta(t)$, we generated the recurrent event $N_i$ from nonhomogeneous Poisson process with mean function $\mu_0(t)\exp(\beta(t)z_i)$. That is, the event number between two consecutive observation times were generated from Poisson distribution with the mean $\mu_0(T_{i,j})\exp(\beta(T_{i,j})z_i)-\mu_0(T_{i,j-1})\exp(\beta(T_{i,j-1})z_i)$ and 
\[
N_{i,j}-N_{i,j-1}\sim Poisson( \mu_0(T_{i,j})\exp(\beta(T_{i,j})z_i)-\mu_0(T_{i,j-1})\exp(\beta(T_{i,j-1})z_i)).
\]

We considered the mean function model under two parameter settings. For each setting, we set $p=1$ and used Epanechnikov kernel to estimate the local kernel estimator with bandwidth $h$ equal to 0$\cdot$3 and 0$\cdot$5, respectively. We performed the simulation with sample sizes 300 and 500. For each setting, we generated 1000 datasets. In this section, we only showed the results under sample size of 300, and the simulation results with sample size of 500 was presented in Appendix 3. We performed the estimation at 100 equally spaced grid points on the time interval. The maximum number of observed times for per individual was $C=10$, and the maximum follow-up time was 6.

In the first setting, we set  the regression function as $\beta(t)=\surd{t}$, and the baseline function $\mu_{0}(t)=2 t^{2}+2$. The results were shown in Figure~\ref{fig1}. Panels a1 and a2 of Figure~\ref{fig1} presented the true curve $\beta(t)$, and the average of the local kernel estimator $\widehat{\beta}(t)$ with the bandwidth set at 0$\cdot$3 and 0$\cdot$5, respectively. The estimators were generally very close to the true value with a slight deviation on the boundary. 
Panels a3 and a4 of Figure~\ref{fig1} compared the estimated and empirical standard errors of the local kernel estimator with bandwidth 0$\cdot$3 and 0$\cdot$5, respectively. As can be seen, there were good agreement between the estimated and empirical standard errors from different bandwidth, with slight bias near the boundary. Panels a5 and a6 of Figure~\ref{fig1} showed the empirical coverage probabilities of the 95$\%$ confidence intervals with bandwidth 0$\cdot$3 and 0$\cdot$5, respectively. The empirical coverage probabilities were generally around 95$\%$ with lower coverage probabilities on the boundary due to the relative larger bias of the coefficient estimator. The simulation results with sample size of 500 showed similar pattern and were displayed in Appendix 3.
Panels a7 and a8 of Figure~\ref{fig1} showed the baseline function estimator. The estimators were close to the true curve. The estimated curve of sample size 500 was closer to the true curve than that of sample size 300 slightly.

\begin{figure}
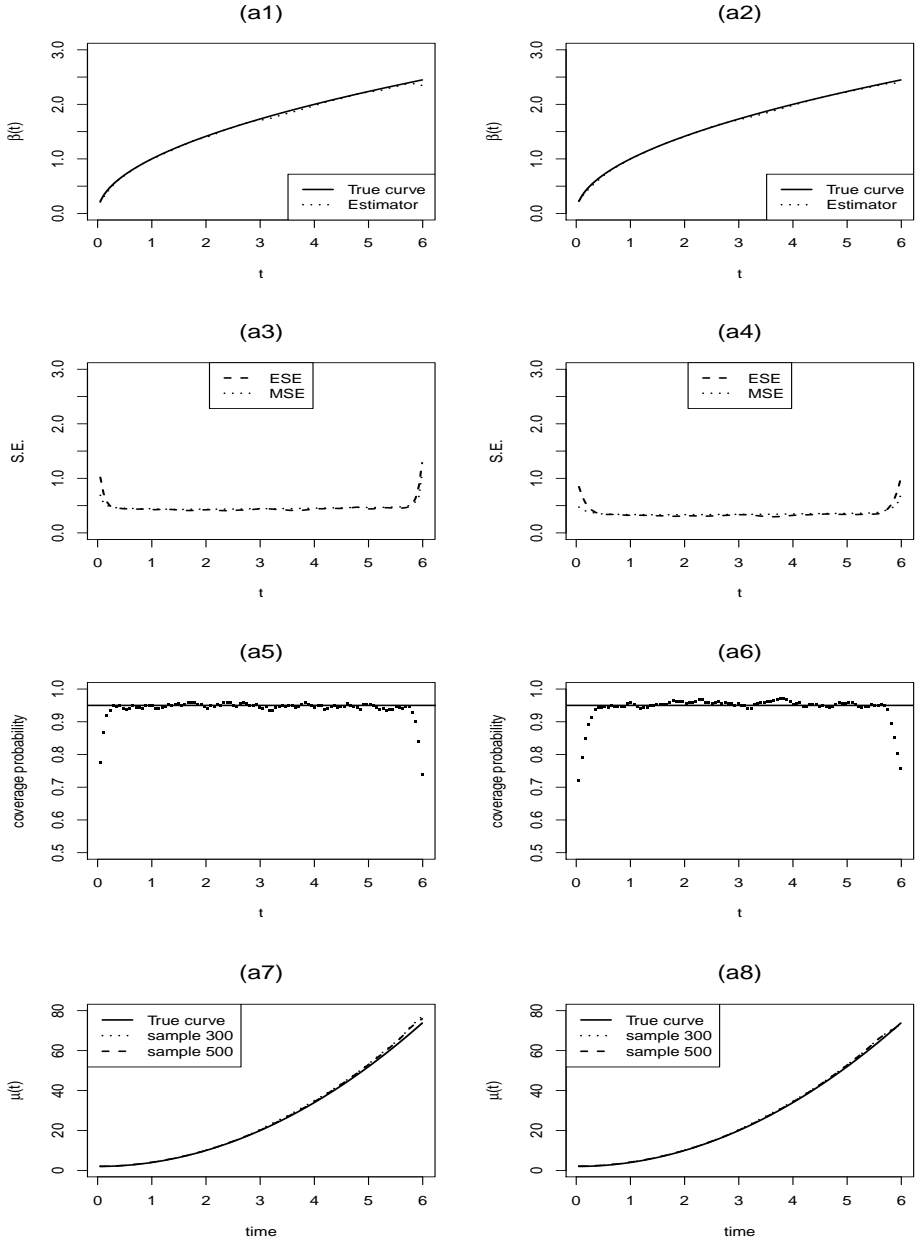

\centering
\begin{minipage}[t]{0.425\linewidth}
\centering
\figurebox{10pc}{\textwidth}{}[be3.3.eps]
\end{minipage}
\begin{minipage}[t]{0.425\textwidth}
\centering
\figurebox{10pc}{\textwidth}{}[be3.5.eps]
\end{minipage}
\begin{minipage}[t]{0.425\textwidth}
\centering
\figurebox{10pc}{\textwidth}{}[se3.3.eps]
\end{minipage}
\begin{minipage}[t]{0.425\textwidth}
\centering
\figurebox{10pc}{\textwidth}{}[se3.5.eps]
\end{minipage}
\begin{minipage}[t]{0.425\textwidth}
\centering
\figurebox{10pc}{\textwidth}{}[cp3.3.eps]
\end{minipage}
\begin{minipage}[t]{0.425\textwidth}
\centering
\figurebox{10pc}{\textwidth}{}[cp3.5.eps]
\end{minipage}
\begin{minipage}[t]{0.425\textwidth}
\centering
\figurebox{10pc}{\textwidth}{}[1mu0.3.eps]
\end{minipage}
\begin{minipage}[t]{0.425\textwidth}
\centering
\figurebox{10pc}{\textwidth}{}[1mu0.5.eps]
\end{minipage}
\caption{(a1) and (a2): The true and the average of the local kernel estimator with bandwidth 0$\cdot$3 and 0$\cdot$5, respectively. (a3) and (a4): Comparison of empirical standard errors (ESE) and the estimated standard errors (MSE) of $\widehat{\beta}(t)$ with bandwidth 0$\cdot$3 and 0$\cdot$5, respectively; (a5) and (a6): Empirical coverage probabilities of the 95$\%$ confidence intervals for $\widehat{\beta}(t)$ with bandwidth 0$\cdot$3 and 0$\cdot$5, respectively. (a7) and (a8): Compare the true baseline curve and the average of the estimator with bandwidth 0$\cdot$3 and 0$\cdot$5, respectively, under sample sizes 300 and 500.}
\label{fig1}
\end{figure}

In the second setting, we set regression function as $\beta(t)$=0$\cdot$5(Beta($t/12,3,3$)+Beta($t/12,4,4$)), where Beta($\cdot$) was the Beta density function, and the baseline function $\mu_{0}(t)=2+ 2\surd{t}$. Similar to the first setting, the results also had good performance as showed in Figure~\ref{fig2}. Panels b1 and b2 of Figure~\ref{fig2} showed that the true curve $\beta(t)$, and the average of the local kernel estimator $\widehat{\beta}(t)$ with the bandwidth set at 0$\cdot$3 and 0$\cdot$5, respectively. The estimators were very close to the true value with a slight bias near the boundary. Panels b3 and b4 of Figure~\ref{fig2} compared the estimated and empirical standard errors of the local kernel estimator with bandwidth equal to 0$\cdot$3 and 0$\cdot$5, respectively. Obviously, there were good concordance between estimated and empirical standard errors, with slight deviation on the boundary. Panels b5 and b6 of Figure~\ref{fig2} displayed the empirical coverage probabilities of the 95$\%$ confidence intervals with bandwidth 0$\cdot$3 and 0$\cdot$5, respectively. The empirical coverage probabilities were generally around 95$\%$. There were lower coverage probabilities on the boundary owing to the relative larger bias of the coefficient estimator. The simulation results with sample size of 500 showed analogous pattern and were presented in Appendix 3.
Panels b7 and b8 of Figure~\ref{fig2} presented the baseline function estimator. The estimators were close to the true baseline curve. Same as the first setting, the estimated curve of sample size 500 was closer to the true curve than that of sample size 300 slightly. 

In summary, the local kernel estimators performed well in terms of small estimation bias and good coverage probabilities of the confidence intervals. We will apply the estimation procedure to analyze a childhood asthma study data. 

\begin{figure}
\centering
\begin{minipage}[t]{0.425\textwidth}
\centering
\figurebox{10pc}{\textwidth}{}[be33.eps]
\end{minipage}
\begin{minipage}[t]{0.425\textwidth}
\centering
\figurebox{10pc}{\textwidth}{}[be35.eps]
\end{minipage}
\begin{minipage}[t]{0.425\textwidth}
\centering
\figurebox{10pc}{\textwidth}{}[se33.eps]
\end{minipage}
\begin{minipage}[t]{0.425\textwidth}
\centering
\figurebox{10pc}{\textwidth}{}[se35.eps]
\end{minipage}
\begin{minipage}[t]{0.425\textwidth}
\centering
\figurebox{10pc}{\textwidth}{}[cp33.eps]
\end{minipage}
\begin{minipage}[t]{0.425\textwidth}
\centering
\figurebox{10pc}{\textwidth}{}[cp35.eps]
\end{minipage}
\begin{minipage}[t]{0.425\textwidth}
\centering
\figurebox{10pc}{\textwidth}{}[2mu0.3.eps]
\end{minipage}
\begin{minipage}[t]{0.425\textwidth}
\centering
\figurebox{10pc}{\textwidth}{}[2mu0.5.eps]
\end{minipage}
\centering
\caption{(b1) and (b2): The true and the average of the local kernel estimator with bandwidth 0$\cdot$3 and 0$\cdot$5, respectively. (b3) and (b4): Comparison of empirical standard errors (ESE) and the estimated standard errors (MSE) of $\widehat{\beta}(t)$ with bandwidth 0$\cdot$3 and 0$\cdot$5, respectively; (b5) and (b6): Empirical coverage probabilities of the 95$\%$ confidence intervals for $\widehat{\beta}(t)$ with bandwidth 0$\cdot$3 and 0$\cdot$5, respectively. (b7) and (b8): Compare the true baseline curve and the average of the estimator with bandwidth 0$\cdot$3 and 0$\cdot$5, respectively, under sample sizes 300 and 500.}\label{fig2}
\end{figure}

\section{Application}

The childhood wheezing study was designed and conducted at Indiana University School of Medicine (\citet{Tepper:2008}). In this study, 105 infants with high risk of developing asthma were recruited. The cumulative wheezing episodes were collected by monthly phone call. The median follow-up time was 33$\cdot$5 months, and the total number of wheezing events was 625. For the baseline characteristics, 49$\cdot$5$\%$ was boys, and 10$\cdot$5$\%$ of children's mothers smoked during pregnancy. And the mean age at enrollment was 10$\cdot$8 months. In recent human asthma study, \citet{JJD:2005} indicated that interleukin-10 (IL-10) regulated the suppressive activity of T cells, which played an important role in human asthma. Furthermore, \citet{HMJ:1998} showed that IL-10 had differential effects on T cells relying on their activated state. The potent anti-inflammatory cytokine IL-10 was shown to be risk factor for infection in early childhood (\citet{Yao:2010}). Also the IL-10's effect may vary during the childhood growth. Therefore, we applied the proposed method to analyze the time-varying effect of interleukin IL-10 about the childhood wheeze data set. 

We estimated the time-varying effect IL-10 on the risk of wheezing using the proposed local kernel estimator. The bandwidth was set as 20 and results were shown in Figure~\ref{fig3}. In general, IL-10 had significant effect on the risk of child wheezing over the follow-up period. The relative risk increased over time period from 25 to 70 months and decreased over time near the boundary. We also estimated the IL-10 effect as a constant coefficient and the overall relative risk was 1$\cdot$53 ($p$-value$<$0$\cdot$05). Although both time-varying and constant effect estimators showed significant results, the time-varying estimator demonstrated an increasing IL-10 effect as age increased. Overall, we illustrated IL-10 was positive associate with child wheezing. And subjects with an increased value of IL-10 will have higher wheezing risk.

\begin{figure}
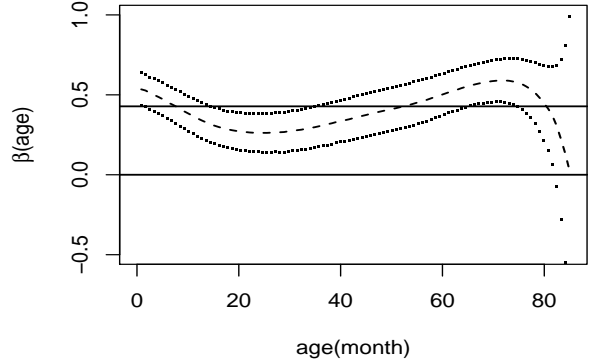

\figurebox{15pc}{20pc}{}[Real20.eps]
\caption{Estimated IL-10 effect, $\widehat{\beta}(age)$, time-varying effect (dash); 95$\%$ confidence interval (dotted); IL-10 effect based on the model with a constant coefficient (horizontal solid line, $\beta$ = 0$\cdot$428).}
\label{fig3}
\end{figure}



\section{Discussion}
In this paper, we propose an local kernel estimation procedure for the panel count model with time-varying coefficients. We construct a  kernel-weighted local partial likelihood at each fixed time point on the basis of local polynomial interpolation. The strong uniform consistency of the proposed estimator is derived. We also show that the proposed estimator is asymptotically normal under some regularity conditions. Furthermore, the simulation results demonstrate the proposed estimation methods perform well under finite sample sizes. The application of the proposed methods for the clinical data analysis also demonstrates that the time-varying coefficient estimation provides more information on the effect of risk factors on the panel count outcome measurement. Through this paper, we provide a nonparametric approach for time-varying coefficient in panel count data. Compared with the spline estimator for panel count model of which the asymptotic normality of $\bm{\widehat{\beta}}$ is not verified, our approach provides a thorough theoretical investigation. The inference of $\bm{\widehat{\beta}}$ is also developed.

Meanwhile, there are challenges remaining in the local kernel method for panel count data, for example, the bandwidth selection. From \eqref{hopt}, the theoretical optimal bandwidth depends on unknown quantities $\sigma_{1}(t)$, $\sigma_{2}(t)$ and $\beta''(t)$. It turned out that bandwidth selection is hard problem for the panel count data model with nonparametric time-varying coefficients, which is beyond the scope of this paper. Nevertheless, there is warrant to select the optimal bandwidth based on data-driven technique. The common selection tool is cross validation method. \citet{Hoover:1998} developed a cross validation criterion of bandwidth selection for longitudinal data. \citet{Cai:2000} proposed a cross validation technique to select optimal bandwidth for time-varying model. \citet{Tian:2005} discussed a $K$-fold cross validation method to select bandwidth for survival data. Therefore, developing a cross validation technique for time-varying coefficients panel count model is of future research interest.

\section*{Acknowledgement}
The research was supported in part by National Natural Science Foundation of China (11671256 Yu), and by Chinese Ministry of Science and Technology (2016YFC0902403 Yu), and by the University of Michigan and Shanghai Jiao Tong University Collaboration Grant (2017, Yu).

\vspace*{-10pt}

\appendix

\appendixone
\section*{Appendix 1}
\subsection*{Notation and Lemmas}

In Appendices, the notations are the same as $\S$\,2 and $\S$\,3. For simplicity of presentation, we introduce some additional notations, let $\bm{\alpha}=H(\bm{\beta}-\bm{\beta}^{\ast})=(\alpha_{0}, \alpha_{1},\ldots, \alpha_{p})^{\T}$, where $\alpha_{k}=h^{k}(\beta_{k}(t)-\beta^{(k)}(t)/k!)$, $H$ is $p$th-order diagonal matrix and $\bm{\beta}^{\ast}$ is the true vector. $\tilde z_i(\bm{u})=H^{-1}z_i(\bm{u})=z_i(1, (u-t)/h,\ldots, (u-t)^{p}/h^{p})^\T$. For a matrix $A=(a_{ij})$, $\|A\|=\sup_{i,j}|a_{ij}|$. For a vector $\bm{a}$, $\|\bm{a}\|=\sup_{i}|a_{i}|$, and $|\bm{a}|=(\sum a_{i}^{2})^{1/2}$.
Some further definitions are:\\
For $j=0,1,2$, set
\begin{align*}
&S_{n,j}(u,\bm{\alpha})=n^{-1} \sum_{i=1}^n I(C_i\geq u)\exp(\bm{\alpha}^\T\tilde z_i(\bm{u})+\bm{\beta}^{\ast \T} z_i(\bm{u})) o_i(u)z_{i}^{j},\\
&S_j(u,\bm{\alpha})=E(p_1(u \mid z)p_2(u \mid z)\exp(\bm{\alpha}^\T \tilde z(\bm{u})+\bm{\beta}^{\ast \T} z(\bm{u}))z^{j});
\end{align*}
For $j=0 ,1, 2$, set
\begin{align*}
&\widetilde S_{n,j}(u)=n^{-1}\sum_{i=1}^n I(C_i \geq u)\exp(\bm{\beta}^{\ast\T} z_i(\bm{u}))o_i(u)\tilde z_i(\bm{u})^{\otimes j},\\
&\widetilde S_j(u)=E(p_1(u \mid z)p_2(u \mid z)\exp(\bm{\beta}^{\ast\T} z_i(\bm{u}))\tilde z(\bm{u})^{\otimes j});
\end{align*}
For $j=0,1$, put
\begin{align*}
&\widetilde S_{n,j}^\ast (u)=n^{-1}\sum_{i=1}^nI(C_i\geq u)exp(\beta(u)z_i)o_i(u)\tilde z_i^{j}(\bm{u}),\\
&\widetilde S_j^\ast(u)=E(p_1(u \mid z)p_2(u \mid z)\exp(\beta(u)z)\tilde z^{j}(\bm{u}));
\end{align*}
For $j=0,1,2$, set
\begin{align*}
&S_{n,j}(u,\bm{\beta}^{\ast})=n^{-1}\sum_{i=1}^n I(C_i\geq u)\exp(\bm{\beta}^{\ast\T}z_i(\bm{u}))o_i(u)z^{j}_i,\\
&S_j(u,\bm{\beta}^{\ast})=E(p_1(u \mid z)p_2(u \mid z) \exp(\bm{\beta}^{\ast\T}z_i(\bm{u}))z^{j});
\end{align*}
For $j=0,1,2$, put
\begin{align*}
&S_{n,j}^\ast(u,\beta(u))=n^{-1}\sum_{i=1}^{n} I(C_i\geq u)\exp(\beta(u)z_i)o_i(u)z_{i}^{j},\\
&S_{j}^\ast(u,\beta(u))=E(p_1(u \mid z)p_2(u \mid z)\exp(\beta(u)z)z^{j}).
\end{align*}
 
\begin{lemma}
\label{lemA1}
Let
\[
c_n(u)=n^{-1}\sum_{i=1}^n I(C_i \geq u) o_i(u)g(u,z_i) \quad and \quad c(u)=E(p_1(u \mid z)p_2(u \mid z) g(u,z)),
\]
if $g(u,z_i)$ is bounded variation, then
\begin{equation}
\label{A1}
\sup_{u \in T}\|c_n(u)-c(u)\|=O_p(n^{-1/2})
\end{equation}
\end{lemma}
\begin{proof}
\label{P1}
Given $g(u,z_i) $ is bounded variation, and under Conditions 2 and 4, we have $o_i(u) g(u,z_i)$ is bounded variation, then we can write $o_i(u) g(u,z_i)=g_1(u,z_i)-g_2(u,z_i)$, where both $g_1(u,z_i)$ and $g_2(u,z_i)$ are nonnegative and nondecreasing. Thus
\begin{equation}
\label{A2}
c_n(u)=n^{-1}\sum_{i=1}^n\{I(C_i\geq u)g_1(u,z_i)-I(C_i\geq u)g_2(u,z_i)\},
\end{equation}
 and $I(C_i \geq u)$, for each $i$, is non-increasing in $u$, then by lemma A.2 of \citet{Bill:1997}, $\{I(C_i \geq u), u \in T\},\{g_j(u, z_i), u \in T\}_{j=1,2}$ have pseudodimension at most 1. By lemma 5.1 of \citet{Pollard:1990} combined with \eqref{A2}, $\{I(C_i \geq u) o_i(u) g(u,z_i),u \in T \}$ has pseduodimension at most 10. Therefore, it must be Euclidean and certainly manageable according to theorem 4.8 of \citet{Pollard:1990}. In view of Condition 2, we choose envelops as $B_1/\surd{n}$, for some constant $B_1$. Then by theorem 8.3 (the uniform laws of large numbers) of \citet{Pollard:1990}, we have $sup_{u \in T}\|c_n(u)-c(u)\|=O_p(n^{-1/2})$.
\end{proof}

\begin{lemma}
\label{lemA2}
Let $T=[a,b]\subset R$, suppose that
\begin{equation}
\label{A3}
\lim_{n \to \infty} \sup_{s \in T} \big \{ |h_n(s)-h(s)| + |J_n(s)-J(s) | \big \}=0,
\end{equation}
where $h_{n}(\cdot)$, $h(\cdot)$ are continuous on $T$, and $J_n(\cdot)$, $J(\cdot)$ are right continuous with bounded variations on $T$. Then
\begin{align}
\label{A4}&\lim_{n\to \infty} \sup_{s \in T}\big \{| \int_{a}^{s} h_n(u)J_n(du)-\int_{a}^{s} h(u)J(du)|\big \}=0,\\
\label{A5}&\lim_{n\to \infty} \sup_{s \in T} \big \{|\int_{a}^{s} h_n(u)J_n(du)-\int_{a}^{s} h_n(u)J(du)| \big \}=0.
\end{align}
\end{lemma}

\begin{proof}
\label{P2}
First, since $h_n$ uniform converges to $h$, and $J_n$ ,$J$ are bounded variation functions with total variations bounded $B_{2}$, for some constant $B_{2}$. Then
\begin{align}
\label{A6}&\lim_{n \to \infty} \sup_{s \in T} \big \{| \int_{a}^{s}h_n(u)J_n(du)-\int_{a}^{s}h(u)J_n(du)| \big \}=0,\\
\label{A7}&\lim_{n \to \infty} \sup_{s \in T} \big \{| \int_{a}^{s}h_n(u)J(du)-\int_{a}^{s}h(u)J(du)|\big \}=0.
\end{align}
Since
\begin{align}
\label{A8}
& |\int_{a}^{s}h_n(u)J_n(du)-\int_{a}^{s}h_n(u)J(du)| \nonumber \\
& \leq | \int_{a}^{s} h_n(u)J_n(du)-\int_{a}^{s}h(u)J(du)|+|\int_{a}^{s}h(u)J(du)-\int_{a}^{s} h_n(u)J(du)|.
\end{align}
Thus, from \eqref{A7} and \eqref{A8}, we know that \eqref{A4} implies \eqref{A5}. And since
\begin{align}
\label{A9}
&| \int_{a}^{s}h_n(u)J_n(du)-\int_{a}^{s}h(u)J(du)| \nonumber\\
&\leq | \int_{a}^{s} h_n(u)J_n(du)-\int_{a}^{s}h(u)J_n(du)| + | \int_{a}^{s}h(u)J_n(du)-\int_{a}^{s}h(u)J(du)|.
\end{align}
For the second term of the right-hand side in \eqref{A9}, since $h(\cdot)$ is continuous, we can partition $T$ by $a=s_0< \ldots <s_{n_0}=b$, and take constant $h_j(=h(s_j))$ such that the simple function:
\begin{equation}
\label{A10}
h_{\varepsilon}(s)=\sum_{j=0}^{n_0-1}h_{j}I(s \in [s_j, s_{j+1}) )
\end{equation}
satisfies
\begin{equation}
\label{A11}
\sup_{s\in T}| h_{\varepsilon}(s)-h(s)| < \varepsilon.
\end{equation}
Thus
\begin{align*}
&| \int_{a}^{s}h(u)J_n(du)-\int_{a}^{s}h(u)J(du)|\\
&\leq | \int_{a}^{s}\{h(u)-h_{\varepsilon}(u)\}J_n(du)|+| \int_{a}^{s}h_{\varepsilon}(u)\{J_n(du)-J(du)\}| + | \int_{a}^{s}\{h(u)-h_{\varepsilon}(u)\}J(du)|\\
&\leq 2\varepsilon B_2+| \int_{a}^{s} \sum_{j=0}^{n_0-1}h_{j}I(u \in [s_j-s_{j+1}))\{J_n(du)-J(du)\}|\\
&=2\varepsilon B_2+|\sum_{j=0}^{n_0-1}h_{j}\int_{s_j}^{s_{j+1}}\{J_n(du)-J(du)\}|\\
&\leq 2\varepsilon B_2+\sum_{j=0}^{n_0-1}| h_{j}| |J_n(s_{j+1})-J(s_{j+1})-J_n(s_j)+J(s_j)|\\
&\leq 2\varepsilon B_2+2\sum_{j=0}^{n_0-1}| h_{j}| \sup_{s\in T}| J_n(s)-J(s)|\\
&\to 2 \varepsilon B_2 \quad as \quad n \to \infty.
\end{align*}
This  in conjunction with \eqref{A6} and \eqref{A9}, we obtain \eqref{A4}. And from \eqref{A4} and \eqref{A7}, then \eqref{A5} holds.
\end{proof}

\vspace*{-10pt}

\appendixtwo
\section*{Appendix 2}
\subsection*{Detailed techniques for theorem proofs}

\begin{proof}[of Theorem~\ref{theorem1}]

This proof is basically same as the proof of Lemma 2.2 of \citet{Hardle:1988} and Theorem 2.1 of \citet{Zhao:1994}. The major difference is that we have to treat a vector parameter $\bm{\beta^{\ast}}=(\beta(t),\beta'(t),\ldots,\beta^{(p)}/p!)^{\T}$ due to the local polynomial estimation. Next up, we will show detailed proof procedure by the below two lemmas. Introduce some notations as follows: 
\begin{equation}
\label{B1}
G_{\alpha_{k}n1}(t, t+s)=n^{-1}\sum_{i=1}^{n}\int_{0}^{\tau}I(C_{i}\geq u)I(t<u<t+s)((u-t)/h)^{k}z_{i}d\widetilde{N}_{i}(u),
\end{equation}
and
\[G_{\alpha_{k}1}(t, t+s)=E(G_{\alpha_{k}n1}(t, t+s));\]
\begin{align}
\label{B2}
&G_{\alpha_{k}n2}(t, t+s)=n^{-1}\sum_{i=1}^{n}\int_{0}^{\tau}I(C_{i}\geq u)I(t<u<t+s)((u-t)/h)^{k}(S_{n,1}(u,\bm{\alpha})/S_{n,0}(u,\bm{\alpha}))\nonumber\\
& d\widetilde{N}_{i}(u),
\end{align}
and
\[
 G_{\alpha_{k}2}(t, t+s)=E(G_{\alpha_{k}n2}(t, t+s));
\]
and for $c> 0$,
\begin{equation}
\label{B3}
V_{\alpha_{k}n1}(t, c)=\sup_{|s| \leq c}|G_{\alpha_{k}n1}(t, t+s)-G_{\alpha_{k}1}(t, t+s)|,
\end{equation}
\begin{equation}
\label{B4}
V_{\alpha_{k}n2}(t, c)=\sup_{|s| \leq c}|G_{\alpha_{k}n2}(t, t+s)-G_{\alpha_{k}2}(t, t+s)|,
\end{equation}
where $S_{n,j}(u,\bm{\alpha}), j=0,1$ defined in Appendix 1, and $\alpha_{k}$ is the $k$th component of $\bm{\alpha}$. Note that $\sup\{\alpha_{k}\}=\bar{\alpha}_{k}$, and $\inf\{\alpha_{k}\}=\underline{\alpha}_{k}$.
\begin{lemma}
\label{lemB1}
Let $0< c_{n} \to 0$, as $n\to \infty$, and $1<c_{n}^{-1}\leq (n/\log n)^{1-2/\lambda}$, then almost surely ($a.s.$),
\begin{equation}
\label{B5}
V_{n1}=\sup_{t\in T}\sup_{\alpha_{k} \in \mathcal{N}_{0}} V_{\alpha_{k}n1}(t, c_{n})=O(n^{-1/2}(c_{n}\log n)^{1/2}), \quad as \quad n\to \infty,
\end{equation}
and
\begin{equation}
\label{B6}
V_{n2}=\sup_{t\in T}\sup_{\alpha_{k} \in \mathcal{N}_{0}} V_{\alpha_{k}n2}(t, c_{n})=O(n^{-1/2}(c_{n}\log n)^{1/2}), \quad as \quad n\to \infty.
\end{equation}
where $\mathcal{N}_{0}:=\{\alpha_{k}:|\alpha_{k}-0|<\epsilon\}$.
\end{lemma}

\begin{proof}
Since $V_{\alpha_{k}n1}$ is a special case of $V_{\alpha_{k}n2}$, when substituted $S_{n,1}(u,\bm{\alpha})/S_{n,0}(u,\bm{\alpha})$ by $z_{i}$. We only need to prove \eqref{B6}. Put
 \[a_{n}=n^{-1/2}(c_{n}\log n)^{1/2}.\]
 As we can treat the positive and negative part of $z_{i}$, separately, we assume that $z_{i}$ is nonnegative. First, we reduce $\sup_{\alpha_{k}\in\mathcal{N}_{0}}$ in \eqref{B6} to a maximum on a finite set. We use finite points $b_{1}<b_{2}<\ldots<b_{N_{n}}$ to partition $\mathcal{N}_{0}$, such that $b_{1}-\underline{\alpha}_{k}\leq a_{n}$, $\bar{\alpha}_{k}-b_{N_{n}}\leq a_{n}$, and $b_{j}-b_{j-1}\leq a_{n}$, for $2\leq j \leq N_{n}$. Further, we assume that 
 \begin{equation}
 \label{B7}
 N_{n}\leq 2(\bar{\alpha}_{k}-\underline{\alpha}_{k})/a_{n},
 \end{equation}
 and for any $t \in T$, and $|s|\leq c_{n}$, by Cauchy-Schwarz inequality, the functions $G_{\alpha_{k}n2}(t, t+s)$ and $G_{\alpha_{k}2}(t, t+s)$ are monotone in $\alpha_{k}$. Letting $J_{n}$ denote the set $\{ \underline{\alpha}_{k},b_{1},\ldots,b_{N_{n}},\bar{\alpha}_{k}\}$, and $J_{n}^{\ast}$ the set $\{(\underline{\alpha}_{k},b_{1}),(b_{1},b_{2}),\ldots,(b_{N_{n}},\bar{\alpha}_{k})\}$. Hence, we have, for any $\alpha_{k}\in \mathcal{N}_{0}$,
 \begin{align*}
&G_{b_{k}n2}(t, t+s)-G_{b_{k}2}(t, t+s)+G_{b_{k}2}(t, t+s)-G_{b_{k+1}2}(t, t+s)\\
& \leq G_{\alpha_{k}n2}(t, t+s)-G_{\alpha_{k}2}(t, t+s)\\
& \leq G_{b_{k+1}n2}(t, t+s)-G_{b_{k+1}2}(t, t+s)+G_{b_{k+1}2}(t, t+s)-G_{b_{k}2}(t, t+s).
\end{align*}
Thus
\begin{align*}
&|G_{\alpha_{k}n2}(t, t+s)-G_{\alpha_{k}2}(t, t+s)|\\
&\leq \max_{\alpha_{k}\in J_{n}}|G_{\alpha_{k}n2}(t, t+s)-G_{\alpha_{k}2}(t, t+s)|+\max_{(\alpha'_{k},\alpha''_{k})\in J_{n}^{\ast}}|G_{\alpha''_{k}2}(t, t+s)-G_{\alpha'_{k}2}(t, t+s)|.
\end{align*}
For $\alpha'_{k}<\alpha''_{k}$,
\begin{align*}
&G_{\alpha''_{k}2}(t, t+s)-G_{\alpha'_{k}2}(t, t+s)|\\
& =|\int_{0}^{\tau}I(t<u<t+s)((u-t)/h)^{k}S_{0}^{\ast}(u,\beta(u))\{S_{2}(u,\bm{\alpha})/S_{0}(u,\bm{\alpha})-(S_{1}(u,\bm{\alpha})/S_{0}(u,\bm{\alpha}))^{2}\}(\alpha''_{k}\\
&-\alpha'_{k})du|\\
&\leq |\int_{0}^{\tau}I(t<u<t+s)((u-t)/h)^{k}M_{0}(\alpha''_{k}-\alpha'_{k})du|\\
&\leq M_{0}a_{n},
\end{align*}
there exists some positive constant $M_{0}$ satisfied the upper inequality, under Conditions 1-5.\\
Hence
\begin{equation}
\label{B8}
V_{n2}\leq \sup_{t\in T}\max_{\alpha_{k}\in J_{n}}V_{\alpha_{k}n2}(t, c_{n})+M_{0}a_{n}.
\end{equation}
Next, we reduce $\sup_{t\in T}$ to a maximum on a finite set.  Now we partition $T$ by an equally-spaced grid $I_{n}:=\{t_{k}: t_{k}=k c_{n}, k=0,\ldots, [\tau/c_{n}]\}$, with $t_{[\tau/c_{n}]+1}=\tau$, where $[\cdot]$ denote the greatest integer part. For any $t\in T$ and $|s|\leq c_{n}$, there exists a grid point $t_{k}$, such that both $t$ and $t+s$ are between $t_{k}$ and $t_{k+1}$. And
\begin{align*}
&|G_{\alpha_{k}n2}(t, t+s)-G_{\alpha_{k}2}(t, t+s)|\\
&\leq |G_{\alpha_{k}n2}(t_{k}, t+s)-G_{\alpha_{k}2}(t_{k}, t+s)|+|G_{\alpha_{k}n2}(t_{k}, t)-G_{\alpha_{k}2}(t_{k}, t)|.
\end{align*}
Then, we obtain
\[
|G_{\alpha_{k}n2}(t, t+s)-G_{\alpha_{k}2}(t, t+s)|\leq 2\max_{t\in I_{n}}V_{\alpha_{k}n2}(t, c_{n}).
\]
Thus 
\begin{equation}
\label{B9}
V_{n2}\leq 2\max_{t\in I_{n}}\max_{\alpha_{k}\in J_{n}}V_{\alpha_{k}n2}(t, c_{n})+2M_{0}a_{n}.
\end{equation}

In order to apply Bernstein's inequality, we truncate $\{z_{i}\}$ by some value, and define $V^{\ast}_{\alpha_{k}n2}(t, c_{n})$ similar to $V_{\alpha_{k}n2}(t, c_{n})$. Put
\[
Q_{n}=c_{n}/a_{n},
\]
and 
\begin{align*}
&G^{\ast}_{\alpha_{k}n2}(t, t+s)=n^{-1}\sum_{i=1}^{n}\int_{0}^{\tau}I(C_{i}\geq u)((u-t)/h)^{k}I(t<u<t+s)\{\sum_{j=1}^{n}I(C_{j}\geq u)\exp(\bm{\alpha}^{\T}\tilde{z}_{j}(\bm{u})\\
&+\bm{\beta}^{\ast \T}z_{j}(\bm{u}))z_{j}I(z_{j}\leq Q_{n})o_{j}(u)/S_{n,0}(u,\bm{\alpha})\}d\widetilde{N}_{i}(u),
\end{align*}
and 
\[
G^{\ast}_{\alpha_{k}2}(t, t+s)=E(G^{\ast}_{\alpha_{k}n2}(t, t+s)).
\]
Likewise, we have
\[
V^{\ast}_{\alpha_{k}n2}(t, c_{n})=\sup_{|s| \leq c_{n}}|G^{\ast}_{\alpha_{k}n2}(t, t+s)-G^{\ast}_{\alpha_{k}2}(t, t+s)|,\]
\[
V^{\ast}_{n2}= \max_{t\in I_{n}}\max_{\alpha_{k}\in J_{n}}V^{\ast}_{\alpha_{k}n2}(t, c_{n}).\]
Thus
\begin{equation}
\label{B10}
V_{n2}\leq V^{\ast}_{n2}+2M_{0}a_{n}+2A_{n1}+2A_{n2},
\end{equation}
where 
\begin{equation*}
A_{n1}=\sup_{t\in I_{n}}\sup_{\alpha_{k}\in J_{n}}\sup_{|s|\leq c_{n}}(G_{\alpha_{k}n2}(t, t+s)-G^{\ast}_{\alpha_{k}n2}(t, t+s)),
\end{equation*}
\begin{equation*}
A_{n2}=\sup_{t\in I_{n}}\sup_{\alpha_{k}\in J_{n}}\sup_{|s|\leq c_{n}}(G_{\alpha_{k}2}(t, t+s)-G^{\ast}_{\alpha_{k}2}(t, t+s)).
\end{equation*}
For
\begin{equation}
\begin{split}
\label{B11}
&G_{\alpha_{k}n2}(t, t+s)-G^{\ast}_{\alpha_{k}n2}(t, t+s))\\
&=n^{-1}\sum_{i=1}^{n}\int_{0}^{\tau}I(C_{i}\geq u)((u-t)/h)^{k}I(t<u<t+s)\{\sum_{j=1}^{n}I(C_{j}\geq u)\exp(\bm{\alpha}^{\T}\tilde{z}_{j}(\bm{u})+\bm{\beta}^{\ast \T}z_{j}(\bm{u}))\\
&z_{j}I(z_{j}> Q_{n})o_{j}(u)/S_{n,0}(u,\bm{\alpha})\}d\widetilde{N}_{i}(u)\\
&\leq Q_{n}^{1-\lambda}n^{-1}\sum_{i=1}^{n}\int_{0}^{\tau}I(C_{i}\geq u)\{\sum_{j=1}^{n}I(C_{j}\geq u)\exp(\bm{\alpha}^{\T}\tilde{z}_{j}(\bm{u})+\bm{\beta}^{\ast \T}z_{j}(\bm{u}))z_{j}^{\lambda}o_{j}(u)/S_{n,0}(u,\bm{\alpha})\}\\
&d\widetilde{N}_{i}(u).
\end{split}
\end{equation}
We have, by the classical strong low of large numbers and Lemma~\ref{lemA1},
\begin{align}
\label{B12}
&n^{-1}\sum_{i=1}^{n}\int_{0}^{\tau}I(C_{i}\geq u)\{\sum_{j=1}^{n}I(C_{j}\geq u)\exp(\bm{\alpha}^{\T}\tilde{z}_{j}(\bm{u})+\bm{\beta}^{\ast \T}z_{j}(\bm{u}))z_{j}^{\lambda}o_{j}(u)/S_{n,0}(u,\bm{\alpha})\}d\widetilde{N}_{i}(u) \to \nonumber \\
&\int_{0}^{\tau}S^{\ast}_{0}(u,\beta(u))E(p_{1}(u \mid z)p_{2}(u \mid z)\exp(\bm{\alpha}^{\T}\tilde{z}(\bm{u})+\bm{\beta}^{\ast\T}z(\bm{u}))z^{\lambda})/S_{0}(u,\bm{\alpha})du<\infty, \quad a.s. 
\end{align} 
Noting that 
\begin{equation}
\label{B13}
a_{n}^{-1}Q_{n}^{1-\lambda}=(c_{n}^{-1}(\log n/n)^{1-2/\lambda})^{\lambda/2}=o(1).
\end{equation}
From \eqref{B11}, \eqref{B12} and \eqref{B13}, we have, as $n\to \infty$,
\begin{equation}
\label{B14}
a_{n}^{-1}A_{n1} \to 0, \quad a.s. 
\end{equation}
From \eqref{B11}, \eqref{B12}, \eqref{B13} and $A_{n2}\leq E(A_{n1})$, then, as $n\to \infty$,
\begin{equation}
\label{B15}
a_{n}^{-1}A_{n2} \to 0, \quad a.s.
\end{equation}
Then, combining \eqref{B10}, \eqref{B14} and \eqref{B15}, it suffices for \eqref{B6} to show
\begin{equation}
\label{B16}
V_{n2}^{\ast}=O(a_{n}) \quad a.s.
\end{equation}
Next we will find a suitable upper bound for $pr(V_{n2}^{\ast}\geq B_{0}a_{n})$ by appropriate choice of $B_{0}$. Now we perform a further partition for $V_{\alpha_{k}n2}^{\ast}(t, c_{n})$ at a fixed $t\in I_{n}$. Set $w_{n}=[(Q_{n}c_{n}/a_{n})+1]$, and $s_{r}=r c_{n}/w_{n}$, for $r= -w_{n},-w_{n}+1,\ldots, w_{n}$. Since $G_{\alpha_{k}n2}^{\ast}(t, t+s)$ and $G_{\alpha_{k}2}^{\ast}(t, t+s)$ are monotone in $|s|$, suppose that $0\leq s_{r}\leq s\leq s_{r+1}$, then
\begin{align*}
&G_{\alpha_{k}n2}^{\ast}(t, t+s_{r})-G_{\alpha_{k}2}^{\ast}(t, t+s_{r})+G_{\alpha_{k}2}^{\ast}(t, t+s_{r})-G_{\alpha_{k}2}^{\ast}(t, t+s_{r+1})\\
&\leq G_{\alpha_{k}n2}^{\ast}(t, t+s)-G_{\alpha_{k}2}^{\ast}(t, t+s)\\
&\leq G_{\alpha_{k}n2}^{\ast}(t, t+s_{r+1})-G_{\alpha_{k}2}^{\ast}(t, t+s_{r+1})+G_{\alpha_{k}2}^{\ast}(t, t+s_{r+1})-G_{\alpha_{k}2}^{\ast}(t, t+s_{r}),
\end{align*}
from which we obtain 
\[
|G_{\alpha_{k}n2}^{\ast}(t, t+s)-G_{\alpha_{k}2}^{\ast}(t, t+s)| \leq \max \{\xi_{n,r}, \xi_{n, r+1} \}+G_{\alpha_{k}2}^{\ast}(t+s_{r}, t+s_{r+1}),
\]
where
\[
\xi_{n,r}=|G_{\alpha_{k}n2}^{\ast}(t, t+s_{r})-G_{\alpha_{k}2}^{\ast}(t, t+s_{r})|.
\]
The same holds for $s_{r}\leq s\leq s_{r+1}\leq 0$. Therefore
\begin{equation}
\label{B17}
V_{\alpha_{k}n2}^{\ast}(t, c_{n})\leq \max_{-w_{n}\leq r \leq w_{n}} \xi_{n, r}+ \max_{-w_{n}\leq r \leq w_{n}-1}G_{\alpha_{k}2}^{\ast}(t+s_{r}, t+s_{r+1}).
\end{equation}
For all $r$, under Condition 5,
\[
G_{\alpha_{k}2}^{\ast}(t+s_{r}, t+s_{r+1}) \leq \int_{t+s_{r}}^{t+s_{r+1}} q_{0}(u)Q_{n}du\leq M_{3}Q_{n}(s_{r+1}-s_{r})\leq M_{3}a_{n},
\]
so that
\begin{equation}
\label{B18}
pr(V_{\alpha_{k}n2}^{\ast}(t, c_{n})\geq B_{0}a_{n})\leq pr(\max_{-w_{n}\leq r \leq w_{n}} \xi_{n, r}\geq (B_{0}-M_{3})a_{n}).
\end{equation}
Now, let 
\begin{align*}
&X_{i}=\int_{0}^{\tau}I(C_{i}\geq u)((u-t)/h)^{k}I(t<u<t+s)\{\sum_{j=1}^{n}I(C_{j}\geq u)\exp(\bm{\alpha}^{\T}\tilde{z}_{j}(\bm{u})+\bm{\beta}^{\ast \T}z_{j}(\bm{u}))z_{j}I(z_{j} \\
&\geq Q_{n})o_{j}(u)/S_{n,0}(u,\bm{\alpha})\}d\widetilde{N}_{i}(u),
\end{align*}
then 
\[
\xi_{nr}=|n^{-1}\sum_{i=1}^{n}\{X_{i}-E(X_{i})\}|.\]
For
\[
|X_{i}-E(X_{i})|\leq |\int_{0}^{\tau}I(C_{i}\geq u)((u-t)/h)^{k}I(t<u<t+s)Q_{n}d\widetilde{N}_{i}(u)|\leq \bar{N}Q_{n},\]
where $\bar{N}=\tau \sup_{u\in T} N_{i}(u)$.\\
And, for some constant $M_{4}$, we have
\begin{align*}
&\sum_{i=1}^{n}var(X_{i}) \leq \sum_{i=1}^{n}E(X_{i}^{2})\\
&\leq \sum_{i=1}^{n} \int_{0}^{\tau}I(t\leq u\leq t+s_{r})((u-t)/h)^{k}E(p_{1}(u \mid z)p_{2}(u \mid z)E(N^{2}(u)\mid z))\{E(p_{1}(u \mid z)p_{2}(u \mid z)\\
&\exp(\bm{\alpha}^{\T}\tilde{z}(\bm{u})+\bm{\beta}^{\ast\T}z(\bm{u}))zI(z\leq Q_{n}))/S_{0}(u,\bm{\alpha})\}^{2}du\\
&\leq \sum_{i=1}^{n}\int_{t+s_{r}}^{t+s_{r+1}}M_{4}du\leq nM_{4}c_{n}.
\end{align*}
Then, by Bernstein's inequality,
\begin{align*}
pr(\xi \geq (B_{0}-M_{3})a_{n}) &\leq \exp \{ -((B_{0}-M_{3})na_{n})^{2}/2(\sum_{i=1}^{n} var(X_{i})+3^{-1}(B_{0}-M_{0})\bar{N}Q_{n}na_{n})\}\\
& \leq \exp \{-((B_{0}-M_{3})na_{n})^{2}/2(M_{4}nc_{n}+3^{-1}(B_{0}-M_{0})\bar{N}Q_{n}na_{n})\}\leq n^{-B_{0}^{\ast}},
\end{align*}
where
\begin{equation}
\label{B19}
B_{0}^{\ast}=(B_{0}-M_{3})^{2}/2(M_{4}+3^{-1}(B_{0}-M_{3})\bar{N}).
\end{equation}
By \eqref{B18} and Boole's inequality,
\begin{equation}
\label{B20}
pr(\sup_{t\in I_{n}}\sup_{\alpha_{k}\in J_{n}}V_{\alpha_{k}n2}^{\ast}(t, c_{n})\geq B_{0}a_{n})\leq (N_{n}+2)([\tau/c_{n}]+1)2[(Q_{n}c_{n}/a_{n})+1]n^{-B_{0}^{\ast}},
\end{equation}
From \eqref{B7}, we obtain
\[
N_{n}+2\leq 2(\bar{\alpha}_{k}-\underline{\alpha}_{k})a_{n}^{-1}+2.
\]
And, obviously,
\[
[\tau/c_{n}]+1\leq(\tau+1)c_{n}^{-1}.
\]
Also,
\[
2[(Q_{n}c_{n}/a_{n})+1]\leq (2Q_{n}c_{n}/a_{n})+2\leq 2((c_{n}a_{n}^{-1})^{2}+1),
\]
since 
\[
(c_{n}a_{n}^{-1})^{2}=c_{n}n/\log n\geq c_{n}^{-2/(\lambda-2)}\geq 1,
\]
then, we have, 
 \[
2[(Q_{n}c_{n}/a_{n})+1]\leq 3c_{n}^{2}a_{n}^{-2}.
\]
Hence
\begin{equation}
\label{B21}
pr(V_{n2}^{\ast}\geq B_{0}a_{n})\leq 2(\bar{\alpha}_{k}-\underline{\alpha}_{k}+1)(\tau+1)3c_{n}a_{n}^{-3}n^{-B_{0}^{\ast}}
\leq \bar{M}_{0}(n/\log n)^{(2\lambda-1/\lambda)}n^{-B_{0}^{\ast}},
\end{equation}
for some constant $\bar{M}_{0}$.

Given $\lambda$ and real $\kappa>0$, we choose a suitable $B_{0}$ denoted as $B_{\kappa,\lambda}$ to make the constant $B_{0}^{\ast}$ in \eqref{B19} satisfies
\[
B_{0}^{\ast}\geq \kappa+(2\lambda -1)/\lambda.
\]
And using $(2\lambda-1)/\lambda=2-1/\lambda>1$, for $\lambda>2$, then \eqref{B21} yields 
\begin{equation}
\label{B22}
pr(V_{n2}^{\ast}\geq B_{\kappa,\lambda}a_{n})\leq \bar{M}_{0}(\log n)^{-1}n^{-\kappa}.
\end{equation}
When $\kappa\geq 2$ in \eqref{B22}, $pr(V_{n2}^{\ast}\geq B_{\kappa,\lambda}a_{n})$ is summable in $n$. So, applying the Borel-Cantelli lemma,
\begin{equation}
\label{B23}
V_{n2}^{\ast}=O(a_{n}), \quad a.s.
\end{equation}
Thus, form \eqref{B10}, \eqref{B14}, \eqref{B15} and \eqref{B23}, we have \[V_{n2}=O(a_{n}),\quad a.s.\]
Similarly, we can also prove $V_{n1}=O(a_{n}), \quad a.s.$
\end{proof}

\begin{lemma}
\label{lemB2}
Let $h$ be a bandwidth and $c_{n}=2h$. Assume that $h\to 0$ and $h^{-1}(\log n/n)^{1-2/\lambda}=o(1)$, let
\begin{equation}
\label{B24}
U_{nk}(\bm{\alpha})=n^{-1}\sum_{i=1}^{n}\int_{0}^{\tau}I(C_{i}\geq u)K_{h}(u-t)((u-t)/h)^{k}\{z_{i}-S_{n,1}(u,\bm{\alpha})/S_{n,0}(u,\bm{\alpha})\}d\widetilde{N}_{i}(u),
\end{equation}
Then we have
\begin{equation}
\label{B25}
\sup_{t\in T}\sup_{\alpha_{k}\in \mathcal{N}_{0}}(nh/\log n)^{1/2}|U_{nk}(\bm{\alpha})-E(U_{nk}(\bm{\alpha}))|=O(1),\quad a.s.
\end{equation}
\end{lemma}

\begin{proof}
Since $K(\cdot)$ is bounded variation function, so we can write $K(\cdot)=K_{1}(\cdot)-K_{2}(\cdot)$, where $K_{1}(\cdot)$ and $K_{2}(\cdot)$ are both increasing functions. Without loss of generality, suppose that $K_{1}(-1)=K_{2}(-1)=0$. Next up, we apply Lemma~\ref{lemB1} by letting $c_{n}=2h$. It is clear that the assumption of Lemma~\ref{lemB1} hold here. Write
\begin{align*}
U_{nk}(\bm{\alpha})&=\int_{-h}^{h}\{n^{-1}\sum_{i=1}^{n}\int_{0}^{\tau}I(C_{i}\geq u)I(v<u-t<h)((u-t)/h)^{k}(z_{i}-S_{n,1}(u,\bm{\alpha})/S_{n,0}(u,\bm{\alpha}))\\
&\quad d\widetilde{N}_{i}(u)\}dK_{h}(v)\\
&=\int_{-h}^{h}\{ G_{\alpha_{k}n1}(t+v, t+h)-G_{\alpha_{k}n2}(t+v, t+h)\}dK_{h}(v),
\end{align*}
where $G_{\alpha_{k}n1}$ and $G_{\alpha_{k}n2}$ defined as \eqref{B1} and \eqref{B2}, respectively. So, we have 
\begin{align*}
\sup_{t\in T}\sup_{\alpha_{k}\in \mathcal{N}_{0}}|U_{nk}(\bm{\alpha})-E(U_{nk}(\bm{\alpha}))|&\leq \sup_{t\in T}\sup_{\alpha_{k}\in \mathcal{N}_{0}} \{V_{\alpha_{k}n1}(t,2h)+V_{\alpha_{k}n2}(t,2h)\}\int_{-h}^{h}dK_{h}(v)\\
&\leq (K_{1}(1)+K_{2}(1))h^{-1}\sup_{t\in T}\sup_{\alpha_{k}\in \mathcal{N}_{0}} \{V_{\alpha_{k}n1}(t,2h)+V_{\alpha_{k}n2}(t,2h)\}.
\end{align*}
Hence, by the consequence of Lemma~\ref{lemB1}, we can derive
\begin{equation}
\label{B26}
\sup_{t\in T}\sup_{\alpha_{k}\in \mathcal{N}_{0}}|U_{nk}(\bm{\alpha})-E(U_{nk}(\bm{\alpha}))|=O((\log n/(nh))^{1/2}), \quad a.s.
\end{equation}
Thus establishing ~\eqref{B25}.
\end{proof}

Next, we will prove Theorem~\ref{theorem1}.
Since $\bm{\alpha}=H(\bm{\beta}-\bm{\beta}^{\ast})$ and $\alpha_{k}(t)=\alpha_{k}=h^{k}(\beta_{k}(t)-\beta^{(k)}(t)/k!)$ defined in Appendix 1, from \eqref{L3}, we have
\begin{equation}
\label{B27}
\mathcal{L}_{n}(\bm{\alpha})=n^{-1}\sum_{i=1}^{n}\int_{0}^{\tau}I(C_{i}\geq u)K_{h}(u-t)\{\bm{\alpha}^{\T}\tilde{z}_{i}(\bm{u})+\bm{\beta}^{\ast \T}z_{i}(\bm{u})-\log S_{n,0}(u,\bm{\alpha})\}d\widetilde{N}_{i}(u),
\end{equation}
and 
\begin{align*}
U_{nk}(\bm{\alpha})&=\partial \mathcal{L}_{n}(\bm{\alpha})/\partial \alpha_{k}\\
&=n^{-1}\sum_{i=1}^{n}\int_{0}^{\tau}I(C_{i}\geq u)K_{h}(u-t)((u-t)/h)^{k}\{z_{i}-S_{n,1}(u,\bm{\alpha})/S_{n,0}(u,\bm{\alpha})\}d\widetilde{N}_{i}(u).
\end{align*}
By the assumption of Condition 3, we have $w(h)=\sup_{|t-t'|\leq h}|\alpha_{k}(t)-\alpha_{k}(t')|=O(h)$.  In this, we consider $\alpha_{k}$ in the neighborhood of zero, that is $\alpha_{k}\in \mathcal{N}_{0}$. And we take $\epsilon=\epsilon_{k}=\max\{2w(h),6l_{n}/(\mu_{2k}M_{1})\}$. Now, we consider $\alpha_{k}\in (-\epsilon_{k},\epsilon_{k})$, without loss of generality, we assume $\epsilon_{k}<1$. Define 
\begin{equation}
\label{B28}
U_{nk}(\epsilon_{k})=n^{-1}\sum_{i=1}^{n}\int_{0}^{\tau}I(C_{i}\geq 0)K_{h}(u-t)((u-t)/h)^{k}(z_{i}-S_{n,1}(\epsilon_{k},u)/S_{n,0}(\epsilon_{k},u))d\widetilde{N}_{i}(u),
\end{equation}
with
\[
S_{n,j}(\epsilon_{k},u)=\sum_{i=1}^{n}I(C_{i}\geq 0)\exp(\epsilon_{k}z_{i}((u-t)/h)^{k}+\bm{\beta}^{\ast \T}z_{i}(\bm{u}))o_{i}(u)z_{i}^{j}, \quad  j=0,1,2.
\]
So by Lemma~\ref{lemB1} and Lemma~\ref{lemB2}, we have (as $n\to \infty$) $a.s.$, for any $ t \in T$,
\begin{equation}
\label{B29}
|U_{n k}(\pm \epsilon_{k})-E(U_{n k}(\pm \epsilon_{k}))|\leq l_{n},
\end{equation}
where $l_{n}=O((\log n/(n h))^{1/2})$.

Under conditions 1-5, and by Lemma~\ref{lemA1}, we have,
\begin{equation}
\label{B30}
E(U_{nk}(\epsilon_{k}))=\int_{0}^{\tau}K_{h}(u-t)((u-t)/h)^{k}\{q_{1}(u)-q_{0}(u)S_{1}(\epsilon_{k},u)/S_{0}(\epsilon_{k},u)\}du,
\end{equation}
where
\begin{align*}
&q_{j}(u)=E(p_{1}(u \mid z)p_{2}(u \mid z)\mu_{0}(u)\exp(\beta(u)z)z^{j}),\quad j=0,1,2.\\
&S_{j}(\epsilon_{k},u)=E(p_{1}(u \mid z)p_{2}(u \mid z)\exp(\epsilon_{k}z((u-t)/h)^{k}+\bm{\beta}^{\ast \T}z(\bm{u}))z^{j}),\quad j=0,1,2.
\end{align*}
Let $(u-t)/h=v$, and $h$ sufficiently small, by Taylor expansion, we have,
\begin{align*}
&E(U_{nk}(\epsilon_{k},u))= \int K(v)v^{k}\{q_{1}(t)-q_{0}(t)E(p_{1}(t \mid z)p_{2}(t \mid z)z\exp(\epsilon_{k}zv^{k}+\beta(t)z))/E(p_{1}(t \mid z)\\
&p_{2}(t\mid z)\exp(\epsilon_{k}zv^{k}+\beta(t)z))\}dv+O(h).
\end{align*}
For
\[
\exp(\epsilon_{k}zv^{k}+\beta(t)z)=\exp(\beta(t)z) \exp(\epsilon_{k}v^{k}z)= \exp(\beta(t)z)(1+\epsilon_{k}v^{k}z+o(\epsilon_{k})).\]
Then
\begin{align*}
E(U_{nk}(\epsilon_{k})) &= \int K(v)v^{k}\{q_{1}(t)-q_{0}(t)(q_{1}(t)+q_{2}(t)\epsilon_{k}v^{k})/(q_{0}(t)+q_{1}(t)\epsilon_{k}v^{k})\}dv+o(\epsilon_{k}).\\
&=-\int K(v)v^{2k}\sigma_{1}(t)\epsilon_{k}/(1+o(\epsilon_{k})+\epsilon_{k}v^{k}q_{1}(t)/q_{0}(t))dv.
\end{align*}
Similarly,
\[E(U_{nk}(-\epsilon_{k}))= \int K(v)v^{2k}\sigma_{1}(t)\epsilon_{k}/(1+o(\epsilon_{k})-\epsilon_{k}v^{k}q_{1}(t)/q_{0}(t))dv.\]
Hence, under Condition 5, we have,
\begin{equation}
\label{B31}
E(U_{nk}(\epsilon_{k}))\leq -3^{-1}\mu_{2k}M_{1}\epsilon_{k},
\end{equation}
and 
\begin{equation}
\label{B32}
E(U_{nk}(-\epsilon_{k}))\geq 3^{-1}\mu_{2k}M_{1}\epsilon_{k}.
\end{equation}
Therefore, combing \eqref{B27}, \eqref{B29} and \eqref{B30}, we obtain that (as $n\to \infty$) $a.s.$, for any $t \in T$,
\[
U_{nk}(\epsilon_{k})\leq l_{n}-3^{-1}\mu_{2k}M_{1}\epsilon_{k}<0,
\]
and
\[ U_{nk}(-\epsilon_{k})\geq -l_{n}+3^{-1}\mu_{2k}M_{1}\epsilon_{k}>0.\]
Then the two above inequalities imply that $a.s.$, for any $t \in T$, there exists $\widehat{\alpha}_{k}(t)=\widehat{\alpha}_{k} \in (-\epsilon_{k},\epsilon_{k})$, such that $U_{nk}(\widehat{\alpha}_{k}(t))=0$, and $\widehat{\alpha}_{k}(t)=h^{k}(\widehat{\beta}_{k}(t)-\beta^{(k)}(t)/k!)$. Thus, we have,
\[\sup_{t\in T}|\widehat{\alpha}_{k}(t)|\leq \epsilon_{k}, \quad a.s.\]
and the above proof follows from $\epsilon_{k}=O((\log n/(nh))^{1/2}+h).$ Hence,
\[
\sup_{t\in T}|\widehat{\beta}_{k}(t)-\beta^{(k)}(t)/k!|=O(h^{-k}\{\log n/(nh))^{1/2}+h\}), \quad a.s.
\]
then \eqref{them1} holds.
\end{proof}

\begin{proof}[of Theorem~\ref{theorem2}]

The proof of the asymptotic normality for the coefficient estimator is basically based on the functional central limit theorem of \citet{Pollard:1990}. Similar to the proof of the Theorem 2.1 of \citet{Bill:1997}, we will first show the asymptotic distribution of stochastic functions by the following lemma, which play a crucial role in the proof of Theorem~\ref{theorem2}.

\begin{lemma}
\label{lemB3}
For any nonzero vector $\bm{a}=(a_{1}, \ldots, a_{p})^{\T}$, let
\begin{align}
\label{B33}&u_1(s)=n^{-1}(nh)^{1/2} \sum_{i=1}^{n}\int_{0}^{s}K_{h}(u-t)\bm{a}^{\T}(\bm{u}-\bm{t})dM_{i}(u),\\
\label{B34}&u_2(s)=n^{-1}(nh)^{1/2}\sum_{i=1}^{n}\int_{0}^{s}K_{h}(u-t)\bm{a}^{\T}(\bm{u}-\bm{t})z_{i}dM_{i}(u),
\end{align}
where \[dM_{i}(u)=I(C_{i}\geq u)\{d\widetilde{N}_{i}(u)-\mu_{0}(u)\exp(\beta(u)z_{i})dO_{i}(u)\}.\]
Under Conditions 1-5, we have $\{u_1(s), s\in T\}$ and $\{u_{2}(s),s \in T\}$ converges in distribution to Gaussian processes $ \xi_{1}$ and $\xi_{2},$ respectively, with continuous sample paths, mean 0 and covariance functions identified by 
\begin{align}
\label{B35}&E(\xi_{1}(s_1)\xi_{1}^{'}(s_2))=\int_{0}^{s_{1}\land s_{2}} h K_{h}^{2}(u-t)(\bm{a}^{\T}(\bm{u}-\bm{t}))^{2}E(p_{1}(u \mid z)p_{2}(u \mid z)\sigma(u \mid z))du,\\
\label{B36}&E(\xi_{2}(s_1)\xi_{2}^{'}(s_2))=\int_{0}^{s_{1}\land s_{2}} h K_{h}^{2}(u-t)(\bm{a}^{\T}(\bm{u}-\bm{t}))^{2}E(p_{1}(u \mid z)p_{2}(u \mid z) z^{2} \sigma(u \mid z))du.
\end{align}
\end{lemma}

\begin{proof}
\label{P3}
Since $u_{1}$ is a special case of $u_{2}$, when we use 1 substitute for $z_{i}$ in \eqref{B31}, we only need to prove the convergence for $u_{2}$.
In order to get the desired convergence, Theorem 10.7 (the functional central limit theorem) of \citet{Pollard:1990} was invoked. Therefore conditions (i)-(v) need to be verified. 

To verify (i), using the lemma A.1 of \citet{Bill:1997}, it suffices to show both $\{\int_{0}^{s} K_{h}(u-t)\bm{a}^{\T}(\bm{u}-\bm{t})I(C_{i} \geq u)z_{i}d\widetilde{N}_{i}(u), s \in T\}$ and $ \{\int_{0}^{s} K_{h}(u-t)\bm{a}^{\T}(\bm{u}-\bm{t})I(C_{i}\geq u) \mu_{0}(u)\exp(\beta(u)z_{i})dO_{i}(u), s \in T \}$ are manageable. Without loss of generality, we assume $\bm{a}^{\T}(\bm{u}-\bm{t})>0$ and $z_{i}>0$. Thus, for each $i$, $\int_{0}^{s}K_{h}(u-t)\bm{a}^{\T}(\bm{u}-\bm{t})I(C_{i}\geq u)z_{i}d\widetilde{N}_{i}(u)$ is nondecreasing in $s$. Then it has pseudodimension at most 1. By Theorem 4.8 of \citet{Pollard:1990}, therefore it must be Euclidean and manageable. Similarly, $\{ \int_{0}^{s}K_{h}(u-t)\bm{a}^{\T}(\bm{u}-\bm{t})I(C_{i}\geq u)\mu_{0}(u)\exp(\beta(u)z_{i})dO_{i}(u), s\in T\}$ are also Euclidean and manageable. Thus (i) holds.

To verify (ii), under Conditions 1-5 and lemma~{\rm \ref{lemA1}},
\begin{align*}
&\lim_{n \to \infty}E(u_{2}(s_{1})u_{2}(s_{2}))\\
&=\lim_{n \to \infty}n^{-1}h\sum_{i=1}^{n}E((\int_{0}^{s_1}K_{h}(u-t)\bm{a}^{\T}(\bm{u}-\bm{t})z_{i}dM_{i}(u))(\int_{0}^{s_2}K_{h}(u-t)\bm{a}^{\T}(\bm{u}-\bm{t})z_{i}dM_{i}(u)))\\
&=\int_{0}^{s_{1}\wedge s_{2}}hK_{h}^{2}(u-t)(\bm{a}^{\T}(\bm{u}-\bm{t}))^{2}E(p_{1}(u \mid z)p_{2}(u \mid z)z^{2}\sigma(u \mid z))du.
\end{align*}
Thus (ii) holds. By the classical multivariate central limit theorem, we obtain that the convergence of finite-dimensional distributions of $u_{2}$ to those of $\xi_{2}$ is straightforward. The latter issue is tightness.

 For (iii), (iv), under Conditions 2 and 3, envelops can be chosen as $ B^{\ast}/\surd{n}$, for some constant $B^{\ast}$. Thus (iii) and (iv) holds.
 
 To test (v), for any $s_{1}, s_{2} \in T$, define 
 \[
 \rho_{n}(s_{1},s_{2})=E(u_{2}(s_{1})-u_{2}(s_{2}))^2, \quad \rho(s_1,s_2)=E(\xi_{2}(s_2)-\xi_{2}(s_1))^2.
 \]
  Here,
 \begin{align*}
 &\rho_{n}(s_1,s_2)=E(u_2(s_2)-u_2(s_1))^{2}\\
&=n^{-1} \sum_{i=1}^{n}E(h (\int_{s_1}^{s_2}K_{h}(u-t)\bm{a}^{\T}(\bm{u}-\bm{t})z_{i}dM_{i}(u))^{2})\\
&=n^{-1} \sum_{i=1}^{n}E(| \int_{s_1}^{s_2}hK_{h}^{2}(u-t)(\bm{a}^{\T}(\bm{u}-\bm{t}))^{2}z_{i}^{2}I(C_{i}\geq u)\mu_{0}^{2}(u)\exp(2\beta(u)z_{i})o_{i}(u)du|).
\end{align*}
Clearly, ${\rho_{n}}$ is equicontinuous on $T$, and $\lim_{n \to \infty} \rho_{n}(s_1,s_2)=\rho(s_1,s_2), \rho$ is pseudometric on $T$. Thus $\rho_{n}$ converges to $\rho$, uniformly on $T$. Furthermore, we set $\{s_1^{n}\},\{s_2^{n}\}$ be any two sequences in $T$, it follows that if $\rho(s_1^{n},s_2^{n})\to 0$, then $\rho_{n}(s_1^{n},s_2^{n})\to 0$. Thus (v) holds.

Therefore, using Theorem 10.7 (the functional central limit theorem) of \citet{Pollard:1990}, we can state $u_{2}$ converges in distribution to Gaussian process on $T$ having continuous sample path. Hence, $\{u_1(s),s\in T\}$ and $\{u_2(s),s\in T\}$ converges in distribution to Gaussian processes $\xi_1$ and $\xi_2$, respectively.
\end{proof}

Now, we prove the Theorem~\ref{theorem2}.
Let $\gamma_n=(nh)^{-1/2}$, $\bm{\alpha}=\gamma_n^{-1} H(\bm{\beta}-\bm{\beta}^{\ast})$, then
\begin{align*}
&X_n(\gamma_n \bm{\alpha},\tau)=n^{-1} \sum_{i=1}^{n}\int_{0}^{\tau} I(C_i \geq u) K_{h}(u-t)[\gamma_n \bm{\alpha}^{\T} \tilde{z}_i(\bm{u})-\log\{\sum_{i=1}^{n} I(C_i \geq u)\exp(\gamma_n \bm{\alpha}^{\T}\tilde{z}_j(\bm{u})\\
&+\bm{\beta}^{\ast \T}z_j(\bm{u}))o_j(u)/ \sum_{i=1}^{n}I(C_i \geq u) \exp(\bm{\beta}^{\ast \T}z_j(\bm{u}))o_j(u)\}]d\widetilde{N}_i(u).
\end{align*}
Let 
\[
I(C_i \geq u)d\widetilde{N}_i(u)=dM_i(u)+I(C_i \geq u)\mu_0(u)\exp(\beta(u)z_i)dO_i(u),\]
then
\begin{equation}
\label{B37}
X_n(\gamma_n\bm{\alpha},\tau)=A_n(\gamma_n\bm{\alpha},\tau)+U_n(\gamma_n\bm{\alpha},\tau),
\end{equation}
where
\begin{align*}
&A_n(\gamma_n\bm{\alpha},\tau)=n^{-1}\sum_{i=1}^{n}\int_0^{\tau}K_h(u-t)[\gamma_n\bm{\alpha}^{\T}\tilde{z}_i(\bm{u})-\log\{ S_{n,0}(u,\gamma_n\bm{\alpha})/\widetilde{S}_{n,0}(u)\}]I(C_i \geq u)\\
&\mu_0(u)\exp(\beta(u)z_i)o_i(u)du,\\
&U_n(\gamma_n\bm{\alpha},\tau)=n^{-1}\sum_{i=0}^{n}\int_0^{\tau}K_h(u-t)[\gamma_n\bm{\alpha}^{\T}\tilde{z}_i(\bm{u})-\log\{S_{n,0}(u,\gamma_n\bm{\alpha})/\widetilde{S}_{n,0}(u)\}]dM_i(u).
\end{align*}
For 
\[
A_n(\gamma_n\bm{\alpha},\tau)=\int_0^{\tau}K_h(u-t)[\widetilde{S}_{n,1}^{\ast}(u)^{\T}\gamma_n \bm{\alpha}-\log\{S_{n,0}(u,\gamma_n\bm{\alpha})/\widetilde{S}_{n,0}(u)\}\widetilde{S}_{n,0}^{\ast}(u)]\mu_0(u)du,
\]
by Taylor expansion of $S_{n,0}(u,\gamma_n\bm{\alpha})$ at $\bm{\alpha}=0$, it follows that
\begin{align*}
&\log\{S_{n,0}(u,\gamma_n\bm{\alpha})/\widetilde{S}_{n,0}(u)\}\\
&=(\widetilde{S}_{n,1}(u)/\widetilde{S}_{n,0}(u))^{\T}\gamma_n\bm{\alpha}+2^{-1}\gamma_n^{2}\bm{\alpha}^{\T}[\widetilde{S}_{n,2}(u)/\widetilde{S}_{n,0}(u)-(\widetilde{S}_{n,1}(u)/\widetilde{S}_{n,0}(u))^{\otimes 2}]\bm{\alpha}+o_p(\gamma_n^{2})\\
&=(\widetilde{S}_1(u)/\widetilde{S}_0(u))^{\T}\gamma_n\bm{\alpha}+2^{-1}\gamma_n^{2}\bm{\alpha}^{\T}\{\widetilde{S}_2(u)/\widetilde{S}_0(u)-(\widetilde{S}_1(u)/\widetilde{S}_0(u))^{\otimes 2}\}\bm{\alpha}+o_p(\gamma_n^{2}).
\end{align*}
Hence 
\[
A_n(\gamma_n\bm{\alpha},\tau) = \gamma_{n} A_{n,1}(\tau)^{\T}\bm{\alpha}-2^{-1} \gamma_n^{2}\bm{\alpha}^{\T}F_{n,1}(\tau)\bm{\alpha}+o_p(\gamma_n^{2}),
\]
where
\begin{align*}
&A_{n,1}(\tau)=\int_0^{\tau}K_h(u-t)\{\widetilde{S}_1^{\ast}(u)-\widetilde{S}_1(u)\widetilde{S}_0^{\ast}(u)/\widetilde{S}_0(u)\}\mu_0(u)du,\\
&F_{n,1}(\tau)=\int_0^{\tau}K_h(u-t)\{\widetilde{S}_2(u)/\widetilde{S}_0(u)-(\widetilde{S}_1(u)/\widetilde{S}_0(u))^{\otimes 2}\}\widetilde{S}^{\ast}_0(u)\mu_0(u)du.
 \end{align*}
For $\mid u-t\mid< ch$, let $ u=t+hv$, under Conditions 1-5, we have
\begin{align*}
&F_{n,1}(\tau)=\int K(v)\{\widetilde{S}_2(t+hv)/\widetilde{S}_0(t+hv)-(\widetilde{S}_1(t+hv)/\widetilde{S}_0(t+hv))^{\otimes 2}\}\widetilde{S}_0^{\ast}(t+hv)\mu_0(t+hv)dv\\
&=\sigma_1(t)\Omega_{1}+o_{p}(1),
\end{align*}
where $ \Omega_{1}=\int K(v)\bm{v}^{\otimes 2}dv$, and $\bm{v}=(1,v,\ldots,v^{p})^{\T}$.\\
Thus 
\begin{equation}
\label{B38}
A_n(\gamma_n \bm{\alpha},\tau)=\gamma_n A_{n,1}(\tau)^{\T}\bm{\alpha}- 2^{-1} \gamma_n^{2}\bm{\alpha}^{\T}\sigma_1(t) \Omega_{1} \bm{\alpha}+o_p(\gamma_n^{2}).
\end{equation}
Similarly, we have
\[
U_{n}(\gamma_{n}\bm{\alpha},\tau)=\gamma_{n}\bm{\alpha}^{\T}U_{n,1}(\tau)-2^{-1} \gamma_{n}^{2}\bm{\alpha}^{\T}F_{n,2}(\tau)\bm{\alpha}+o_{p}(\gamma_{n}^{2}),
\]
where
\begin{align*}
&U_{n,1}(\tau)=\int_0^{\tau}K_h(u-t)n^{-1}\sum_{i=1}^{n}\{\tilde{z}_i(\bm{u})-\widetilde{S}_{n,1}(u)/\widetilde{S}_{n,0}(u)\}dM_i(u),\\
&F_{n,2}(\tau)=n^{-1}\sum_{i=1}^{n}\int_0^{\tau}K_h(u-t)\{\widetilde{S}_{n,2}(u)/\widetilde{S}_{n,0}(u)-(\widetilde{S}_{n,1}(u)/\widetilde{S}_{n,0}(u))^{\otimes 2}\}dM_i(u).
\end{align*}
For $F_{n,2}(\tau)$, similar to Lemma~\ref{lemB3}, we have $\{ \int_0^{s}K_h(u-t)dM_i(u), s\in T\}$ is manageable. Let constant $\bar{B}/\surd{n}$ as envelope. Thus, using Theorem 8.3 (the uniform law of large numbers) of \citet{Pollard:1990}, we can derive
\[
\lim_{n\to \infty} \sup_{s\in T} \|n^{-1} \sum_{i=1}^{n}\int_0^{s}K_h(u-t)dM_i(u)-0\|=0.
\]
Also, by Lemma~\ref{lemA1}, as $n\to \infty$,
\[
\sup_{s\in T}\| \{\widetilde{S}_{n,2}(u)/\widetilde{S}_{n,0}(u)-(\widetilde{S}_{n,1}(u)/\widetilde{S}_{n,0}(u))^{\otimes 2}\}-\{\widetilde{S}_2(u)/\widetilde{S}_0(u)-(\widetilde{S}_1(u)/\widetilde{S}_0(u))^{\otimes 2}\}\| \rightarrow 0.
\]
Then, by lemma~{\rm \ref{lemA2}}, we have 
\[
F_{n,2}(\tau)=O_p(\gamma_n),
\]
Therefore, 
\begin{equation}
\label{B39}
U_n(\gamma_n\bm{\alpha},\tau)=\gamma_n\bm{\alpha}^{\T}U_{n,1}(\tau)+O_p(\gamma_n^{2}).
\end{equation}
From \eqref{B37}, \eqref{B38} and \eqref{B39}, we obtain
\[
X_{n}(\gamma_n\bm{\alpha},\tau)=\{A_{n,1}(\tau)+U_{n,1}(\tau)\}^{\T}\gamma_n\bm{\alpha}-2^{-1} \gamma_n^{2}\bm{\alpha}^{\T}\sigma_1(t)\Omega_{1} \bm{\alpha}+o_p(\gamma_n^{2}).
\]
Using Quadratic Approximation Lemma of \citet{FanGij:1996}, we derive
\begin{equation}
\label{B40}
\bm{\widehat{\alpha}}=\gamma_n^{-1}(\sigma_1(t)\Omega_{1})^{-1}\{A_{n,1}(\tau)+U_{n,1}(\tau)\}+o_p(1).
\end{equation}
For
$A_{n,1}(\tau)=\int_0^{\tau}K_h(u-t)\{\widetilde{S}_1^{\ast}(u)-\widetilde{S}_1(u)\widetilde{S}_0^{\ast}(u)/\widetilde{S}_0(u)\}\mu_0(u)du$. We apply Taylor expansion to the term:
\[
 \widetilde{S}_1^{\ast}(u)-\widetilde{S}_1(u)\widetilde{S}_0^{\ast}(u)/\widetilde{S}_0(u)=\widetilde{S}_1^{\ast}(u)-\widetilde{S}_1(u)-\widetilde{S}_1(u)(\widetilde{S}_0^{\ast}(u)-\widetilde{S}_0(u))/\widetilde{S}_0(u).
 \]
Note that 
\begin{align*}
\beta(u)z & \approx \beta(t)z+\beta'(t)z(u-t)+\cdots+\beta^{(p)}(t)z(u-t)^{p}/p!+\beta^{(p+1)}(t)z(u-t)^{p+1}/(p+1)!\\
&=\bm{\beta}^{\ast \T}z(\bm{u})+\beta^{(p+1)}(t)z(u-t)^{p+1}/(p+1)!.
\end{align*}
Then
\begin{align*}
\exp(\beta(u)z)-\exp(\bm{\beta}^{\ast \T}z(\bm{u}))& \approx \exp(\beta(u)z)\{1-\exp(-\beta^{(p+1)}(t)z(u-t)^{p+1}/(p+1)!)\}\\
& \approx \exp(\beta(u)z)\beta^{(p+1)}(t)z(u-t)^{p+1}/(p+1)!.
\end{align*}
Thus
\begin{align*}
&\widetilde{S}_1^{\ast}(u)-\widetilde{S}_1(u)= E(p_1(u \mid z)p_2(u \mid z)\exp(\beta(u)z)z\tilde{z}(\bm{u}))\beta^{(p+1)}(t)(u-t)^{p+1}/(p+1)!\\
&\qquad\qquad \qquad\quad+o((u-t)^{p+1}),\\
&\widetilde{S}_0^{\ast}(u)-\widetilde{S}_0(u)=E(p_1(u \mid z)p_2(u \mid z)\exp(\beta(u)z)z)\beta^{(p+1)}(t)(u-t)^{p+1}/(p+1)!+o((u-t)^{p+1}),\\
&\widetilde{S}_0(u)=E(p_1(u \mid z)p_2(u \mid z)\exp(\beta(u)z))+O((u-t)^{p+1}),\\
&\widetilde{S}_1(u)=E(p_1(u \mid z)p_2(u \mid z)\exp(\beta(u)z)\tilde{z}(\bm{u}))+O((u-t)^{p+1}).
\end{align*}
Therefore, we have
\begin{align*}
&A_{n,1}(\tau)=\int_{0}^{\tau} K_{h}(u-t)[\{E(p_1(t \mid z)p_2(t \mid z)\exp(\beta(u)z)z\tilde{z}(\bm{u}))-\widetilde{S}_1^{\ast}(u)S_1^{\ast}(u,\beta(u))/S_0^{\ast}(u,\beta(u))\}\\
&\beta^{(p+1)}(t)(u-t)^{p+1}/(p+1)!+o((u-t)^{p+1})]du.
\end{align*}
Let $u = t + h v$, we derive
\begin{equation}
\label{B41}
A_{n,1}(\tau)=\int K(v)\bm{v}v^{p+1}dv\sigma_{1}(t)h^{p+1}\beta^{(p+1)}(t)/(p+1)!+o(h^{p+1}).
\end{equation}
From \eqref{B40} and \eqref{B41}, we obtain (let $\bm{b}=\int K(v)v^{p+1}\bm{v}dv$)
\[
\bm{\widehat{\alpha}}=\gamma_n^{-1}\Omega_{1}^{-1}\bm{b}h^{p+1}\beta^{(p+1)}(t)/(p+1)!+\gamma_n^{-1}\sigma_{1}^{-1}(t)\Omega_{1}^{-1}U_{n,1}(\tau)+o_p(1).\]
Hence,
\begin{equation}
\label{B42}
(nh)^{1/2}\{H(\bm{\widehat{\beta}}-\bm{\beta}^{\ast})-\Omega_{1}^{-1}\bm{b}h^{p+1}\beta^{(p+1)}(t)/(p+1)!\}=\gamma_n^{-1}\sigma_{1}^{-1}(t)\Omega_{1}^{-1}U_{n,1}(\tau)+o_p(1).
\end{equation}
Therefore, \eqref{B40} can be reduced to prove the multivariate normality of $(nh)^{1/2}U_{n,1}(\tau)$. That is equivalent to prove the normality of $\bm{a}^{\T}(nh)^{1/2}U_{n,1}(\tau)$, for any nonzero vector $\bm{a}=(a_{1},\ldots, a_{p})^{\T}$. Write $\widetilde{U}_{n}(s)=\bm{a}^{\T}(nh)^{1/2}U_{n,1}(s)$ is empirical process, we will show that it converges to Gaussian process $\tilde{\xi}$.\\
In fact,
\[
\widetilde{U}_{n}(s)=\widetilde{U}_{n1}(s)+\widetilde{U}_{n2}(s),
\]
where
\begin{align*}
&\widetilde{U}_{n1}(s)=n^{-1}(nh)^{1/2}\sum_{i=1}^{n}\int_0^{s}K_h(u-t)\bm{a}^{\T}(\bm{u}-\bm{t})\{z_i-S_{1}(u,\bm{\beta}^{\ast})/S_{0}(u,\bm{\beta}^{\ast})\}dM_i(u),\\
&\widetilde{U}_{n2}(s)=n^{-1}(nh)^{1/2}\sum_{i=1}^{n}\int_0^{s}K_h(u-t)\bm{a}^{\T}(\bm{u}-\bm{t})\{S_{1}(u,\bm{\beta}^{\ast})/S_{0}(u,\bm{\beta}^{\ast})-S_{n,1}(u,\bm{\beta}^{\ast})/S_{n,0}(u,\bm{\beta}^{\ast})\}\\
&\qquad \qquad dM_i(u).
\end{align*}
For $\widetilde{U}_{n2}(s)$, by Lemma~\ref{lemB3} and the Strong Representation Theorem of \citet{Pollard:1990}, we  can construct a new probability space, and have
\[ 
\sup_{s \in T} \| u_1(s)-\xi_1(s) \| \to 0, \quad as \quad n \to \infty,
\]
and by Lemma~\ref{lemA1}, we have 
\[
\sup_{s\in T}\| S_1(u,\bm{\beta}^{\ast})/S_0(u,\bm{\beta}^{\ast})-S_{n,1}(u,\bm{\beta}^{\ast})/S_{n,0}(u,\bm{\beta}^{\ast}) \|\to 0, \quad as \quad n \to \infty.
\]
Then by Lemma~\ref{lemA2}, we can show that almost surely,
\begin{align*}
&n^{-1}(nh)^{1/2} \sum_{i=1}^{n}\int_0^{s}K_h(u-t)\bm{a}^{\T}(\bm{u}-\bm{t})\{S_{1}(u,\bm{\beta}^{\ast})/S_{0}(u,\bm{\beta}^{\ast})-S_{n,1}(u,\bm{\beta}^{\ast})/S_{n,0}(u,\bm{\beta}^{\ast})\}dM_i(u)\\
&\to 0, \quad as \quad n \to \infty.
\end{align*}
which holds in original probability space since the statement is now in probability. Thus the convergence of $\widetilde{U}_{n}(s)$ reduces to that of $\widetilde{U}_{n1}(s)$. Here,
\begin{align*}
&\lim_{n \to \infty}E(\widetilde{U}_{n1}(s_{1})\widetilde{U}_{n1}(s_{2}))\\
&=\lim_{n\to\infty}n^{-1} \sum_{i=1}^{n}E(h(\int_0^{s_1}K_h(u-t)\bm{a}^{\T}(\bm{u}-\bm{t})\{z_i-S_1(u,\bm{\beta}^{\ast})/S_0(u,\bm{\beta}^{\ast})\}dM_i(u))(\int_0^{s_2}K_h(u-t)\\
&\quad \bm{a}^{\T}(\bm{u}-\bm{t})\{z_i-S_1(u,\bm{\beta}^{\ast})/S_0(u,\bm{\beta}^{\ast})\}dM_i(u)))\\
&=\int_0^{s_1 \land s_2} h K_h^{2}(u-t)(\bm{a}^{\T}(\bm{u}-\bm{t}))^{2}E(p_1(u \mid z)p_2(u \mid z)\{z-S_{1}(u,\bm{\beta}^{\ast})/S_{0}(u,\bm{\beta}^{\ast})\}^{2}\mu_{0}^{2}(u)\exp(2\\
& \quad \beta(u)z))du\\
&=E(\tilde{\xi}(s_1)\tilde{\xi}(s_2)).
\end{align*}
Then, the convergence of finite-dimensional distributions of $\widetilde{U}_{n1}(s)$ to those of $\tilde{\xi}$ is clearly true by the classical multivariate central limit theorem, since $\widetilde{U}_{n1}$ is a sum of independent random variables. It remains to show tightness for $\widetilde{U}_{n1}$, or equivalently, tightness for 
\[
\widetilde{U}_{n1}(s)=n^{-1}(nh)^{1/2}\sum_{i=1}^{n}\int_0^{s}K_h(u-t)\bm{a}^{\T}(\bm{u}-\bm{t})\{z_i-S_{1}(u,\bm{\beta}^{\ast})/S_{0}(u,\bm{\beta}^{\ast})\}dM_i(u).
\]
By Lemma~\ref{lemB3}, $\{n^{-1}(nh)^{1/2}\sum_{i=1}^{n}\int_0^{s}K_h(u-t)\bm{a}^{\T}(\bm{u}-\bm{t}) z_idM_i(u), s\in T\}$ is tightness. And analogous to the proof of Lemma~\ref{lemB3}, we can check that $\{n^{-1}(nh)^{1/2}\sum_{i=1}^{n}\int_0^{s}K_h(u-t)\bm{a}^{\T}(\bm{u}-\bm{t})S_1(u,\bm{\beta}^{\ast})/S_0(u,\bm{\beta}^{\ast})dM_i(u), s\in T\}$ is tightness, too. Therefore, $\widetilde{U}_{n1}(s)$ converges to $\tilde{\xi}$. Hence, $\bm{a}^{\T}(nh)^{1/2}{U}_{n,1}(\tau)$ is normal. Then, $(nh)^{1/2}{U}_{n,1}(\tau)$ is multivariate normal, and asymptotically covariance is as follows:
\[
\Sigma_2(t)=\int K^{2}(v)\bm{v}^{\otimes 2}dvE(p_{1}(t \mid z)p_{2}(t \mid z)\mu_{0}^{2}(t)\exp(2\beta(t)z)(z-q_{1}(t)/q_{0}(t))^{2})=\sigma_{2}(t)\Omega_{2},
\]
where $\Omega_{2}=\int K^{2}(v)\bm{v}^{\otimes 2}dv$.\\
Therefore,
\[
 (nh)^{1/2}\{H(\bm{\widehat{\beta}}-\bm{\beta}^{\ast})-\Omega_{1}^{-1}\bm{b}h^{p+1}\beta^{(p+1)}(t)/(p+1)!\}\to N(0,\sigma_1^{-2}(t)\sigma_2(t)\Omega_{1}^{-1}\Omega_{2}\Omega_{1}^{-1}), 
 \]
 as $n\to\infty$, $h\to0$, $nh\to \infty$.
\end{proof}

\begin{proof}[of consistency of covariance]

For $\widehat{\Sigma}(t)=\widehat{\Sigma}_{1}^{-1}(t)\widehat{\Sigma}_{2}(t)\widehat{\Sigma}_{1}^{-1}(t),$ where $\widehat{\Sigma}_{1}(t)$ and $\widehat{\Sigma}_{2}(t)$ defined as \eqref{Sig1} and \eqref{Sig2}, respectively. Next up, we will show that $\widehat{\Sigma}_{1}(t)$ and $\widehat{\Sigma}_{2}(t)$ are consistent, respectively.

First of all, we give a conclusion by the following demonstration. Under Conditions 2-4, there exists a neighborhood $\mathcal{B}$ of $\bm{\beta}^{\ast}$, such that functions $S_j(u,\bm{\beta}), j=0,1, 2$ are continuous in $\bm{\beta} \in \mathcal{B}$, uniformly in $u \in T$. And $S_0(u,\bm{\beta})$ is bounded away of from zero on $(u,\bm{\beta}) \in T \times \mathcal{B}$. Furthermore, by Lemma~\ref{lemA1}, we can derive, for  each $j= 0,1, 2$,
\begin{equation}
\label{SU}
\sup_{\mathcal{B}\times T } \|S_{n,j}(u,\bm{\beta})-S_j(u,\bm{\beta}) \|\to 0,\quad as \quad n \to \infty.
 \end{equation}
 
 Further, account for $\widehat{\Sigma}_{1}(t)$. We will prove $\widehat{\Sigma}_{1}(t)$ converges to $\Sigma_{1}(t)=\sigma_{1}(t)\Omega_{1}$. Let 
 \begin{equation}
 \label{v1}
 v_{1}(u,\beta(u))=S_2^{\ast}(u,\beta(u))/S_0^{\ast}(u,\beta(u))-S_1^{\ast 2}(u,\beta(u))/S_0^{\ast 2}(u,\beta(u)),
 \end{equation}
  and from the defined $q_{j}(t)$, we have $S_j^{\ast}(t,\beta(t))=q_j(t)/\mu_0(t)$, $j=0,1,2$. \\
 Then, we obtain
 \begin{align}
 \label{sig1}
  \Sigma_1(t) & =\{q_{2}(t)-q_1^{2}(t)/q_0(t)\}\Omega_{1}=\mu_0(t)\{S_2^{\ast}(t,\beta(t))-S_1^{\ast 2}(t,\beta(t))/S_0^{\ast}(t,\beta(t))\}\Omega_{1}\nonumber\\
   & =\int K_h(u-t)(\bm{u}-\bm{t})^{\otimes 2}\mu_0(u)S_0^{\ast}(u,\beta(u))v_{1}(u,\beta(u))du+o(1). 
  \end{align}
Using triangle inequality, we have 
\begin{align*}
&\|\widehat{\Sigma}_1(t)-\Sigma_1(t)\|\\
&\leq \| n^{-1}\sum_{i=1}^{n}\int_0^{\tau}I(C_i\geq u)K_h(u-t)(\bm{u}-\bm{t})^{\otimes 2}(V_{1}(u,\bm{\widehat{\beta}})-v_{1}(u,\beta(u)))d\widetilde{N}_i(u)\|\\
&+\| n^{-1}\sum_{i=1}^{n}\int_0^{\tau}I(C_i\geq u)K_h(u-t)(\bm{u}-\bm{t})^{\otimes 2}v_{1}(u,\beta(u))\{d\widetilde{N}_i(u)-\mu_0(u)\exp(\beta(u)z_i)o_i(u)du\}\|\\
&+\|n^{-1}\sum_{i=1}^{n}\int_0^{\tau}I(C_i\geq u)K_h(u-t)(\bm{u}-\bm{t})^{\otimes 2}v_{1}(u,\beta(u))\mu_0(u)\exp(\beta(u)z_i)o_i(u)du-\int_0^{\tau} K_h(u\\
&-t)(\bm{u}-\bm{t})^{\otimes 2}\mu_0(u)S_0^{\ast}(u,\beta(u))v_{1}(u,\beta(u))du\|\\
&+\|\int_0^{\tau} K_h(u-t)(\bm{u}-\bm{t})^{\otimes 2}\mu_0(u)S_0^{\ast}(u,\beta(u))v_{1}(u,\beta(u))du- \mu_0(t)\{S_2^{\ast}(t,\beta(t))-S_1^{\ast 2}(t,\beta(t))/\\
&S_0^{\ast}(t,\beta(t))\}\Omega_{1}\|.
\end{align*}
where $V_{1}(u,\bm{\widehat{\beta}})$ defined in \eqref{V1}.

 For the first term of the right-hand side, under Conditions 1-5, by the consequence of Theorem~\ref{theorem1}, from \eqref{a1}, we can derive
 \begin{equation}
 \label{SS}
  \sup_{\mathcal{B}\times T} \|S_{j}(u,\bm{\widehat{\beta}})-S_{j}^{\ast}(u,\beta(u))\|\to 0, \quad as \quad n\to \infty.
 \end{equation}
 Hence, from \eqref{SU} and \eqref{SS}, we have
  \begin{equation}
  \label{Vv1}
  \sup_{\mathcal{B}\times T}\| V_{1}(u,\bm{\widehat{\beta}})-v_{1}(u,\beta(u))\| \to 0,\quad as \quad n\to\infty.
  \end{equation}
  By consequence of Lenglart inequality,
 \begin{align}
 \label{Len}
& pr(\{n^{-1}\sum_{i=1}^{n}\int_0^{\tau}I(C_i\geq u)K_h(u-t)d\widetilde{N}_i(u)>C\})\nonumber\\
& \leq \frac{\delta}{C}+pr(\{ \int_0^{\tau}n^{-1}\sum_{i=1}^{n}I(C_i\geq u)K_h(u-t)\mu_0(u)\exp(\beta(u)z_i)o_i(u)du>\delta \}),
 \end{align}
when $\delta>\int_0^{\tau}K_h(u-t)\mu_0(u)S_0^{\ast}(u,\beta(u))du=\mu_0(t)S_0^{\ast}(t,\beta(t))$, the latter probability tends to zero as $n\to\infty, h\to 0$, and $ nh \to \infty$. Thus the first term converges to zero.

For the second term of the right-hand side, that is $n^{-1}\sum_{i=1}^{n}\int_0^{\tau}K_h(u-t)v_{1}(u,\beta(u))dM_i(u)$ is empirical process. by Lemma~\ref{lemB3} and $v_{1}(u,\beta(u))$ is non-negative function, analogous to the proof of Theorem~\ref{theorem2}, using the Theorem 8.3 (the uniform law of large numbers) of \citet{Pollard:1990}, we can demonstrate the second term converges to zero.

For the third term of the right-hand side, under Conditions 1-4, functions $v_{1}(u,\beta(u))$ are bounded. So from \eqref{SU}, it is easy to prove the third term tend to zero.

For the fourth term of the right-hand side, from \eqref{sig1}, obviously, it converges to zero.\\
Therefore, 
\begin{equation}
\label{SS1}
\|\widehat{\Sigma}_1(t)-\Sigma_1(t)\|\to 0,\quad as \quad n\to\infty.
\end{equation}

Next, we will prove $\widehat{\Sigma}_{2}(t)$ converges to $\Sigma_{2}(t)=\sigma_{2}(t)\Omega_{2}$ by the following demonstration. Let
\begin{equation}
\label{v2}
v_{2}(u,\beta(u))=\{z-S_{1}^{\ast}(u,\beta(u))/S^{\ast}_{0}(u,\beta(u))\}^{2}.
\end{equation}
For
\begin{align}
\label{sig2}
\Sigma_2(t)&=E(p_{1}(t \mid z)p_{2}(t \mid z)\mu_{0}^{2}(t)\exp(2\beta(t)z)(z-q_{1}(t)/q_{0}(t))^{2})\Omega_{2}\nonumber\\
 &=\int_0^{\tau}h K_h^{2}(u-t)(\bm{u}-\bm{t})^{\otimes 2}\mu_{0}^{2}(u)E(p_{1}(u \mid z)p_{2}(u \mid z)\exp(2\beta(u)z)v_{2}(u,\beta(u))du+o(1). 
\end{align}
Then, using triangle inequality, we have
\begin{align*}
 & \| \widehat{\Sigma}_2(t)-\Sigma_2(t)\| \\
 &\leq \|n^{-1}\sum_{i=1}^{n} \int_0^{\tau} hK_h^{2}(u-t)(\bm{u}-\bm{t})^{\otimes 2}I(C_{i}\geq u)V_{2}(u,\bm{\widehat{\beta}}) \widehat{\mu}^{2}_0(u,\widehat{\beta}(u))\exp(2\widehat{\beta}(t)z_i)o_{i}(u)du-\int_0^{\tau}\\
 &hK_h^{2}(u-t)(\bm{u}-\bm{t})^{\otimes 2}\widehat{\mu}^{2}_0(u,\widehat{\beta}(u))E(p_{1}(u\mid z)p_{2}(u\mid z)v_{2}(u,\beta(u))\exp(2\beta(t)z))du\|\\
 &+\|\int_0^{\tau} hK_h^{2}(u-t)(\bm{u}-\bm{t})^{\otimes 2}\{\widehat{\mu}^{2}_0(u,\widehat{\beta}(u))-\mu^{2}_{0}(u)\}E(p_{1}(u\mid z)p_{2}(u\mid z)v_{2}(u,\beta(u))\exp(2\beta(t)z))\\
 &du\|\\
 &+\|\int_0^{\tau} hK_h^{2}(u-t)(\bm{u}-\bm{t})^{\otimes 2}\mu^{2}_{0}(u)E(p_{1}(u\mid z)p_{2}(u\mid z)v_{2}(u,\beta(u))\exp(2\beta(t)z))du-E(p_{1}(t \mid z)\\
 &p_{2}(t \mid z)\mu_{0}^{2}(t)\exp(2\beta(t)z)(z-q_{1}(t)/q_{0}(t))^{2})\Omega_{2}\|.
\end{align*}
where $V_{2}(u,\bm{\widehat{\beta}})$ defined in \eqref{V2}.

For the first term of the right-hand side, let $V_{21}(u,\bm{\widehat{\beta}})=S_{n,1}(u,\bm{\widehat{\beta}})/S_{n,0}(u,\bm{\widehat{\beta}}),$ and $ v_{21}=S_{1}^{\ast}(u,\beta(u))/S^{\ast}_{0}(u,\beta(u)).$ We have
\begin{align}
\label{Sig11}
&n^{-1}\sum_{i=1}^{n} \int_0^{\tau} hK_h^{2}(u-t)(\bm{u}-\bm{t})^{\otimes 2}I(C_{i}\geq u)V_{2}(u,\bm{\widehat{\beta}}) \widehat{\mu}^{2}_0(u,\widehat{\beta}(u))\exp(2\widehat{\beta}(t)z_i)o_{i}(u)du\nonumber\\
&=\int_0^{\tau} hK_h^{2}(u-t)(\bm{u}-\bm{t})^{\otimes 2}\widehat{\mu}^{2}_0(u,\widehat{\beta}(u))\{S_{n,2}^{\ast}(u,2\widehat{\beta}(t))-2S_{n,1}^{\ast}(u,2\widehat{\beta}(t))V_{21}(u,\bm{\widehat{\beta}})+S_{n,0}^{\ast}(u,2\widehat{\beta}(t))\nonumber\\
&V_{21}^{2}(u,\bm{\widehat{\beta}})\}du.
\end{align}
From \eqref{them1.1} and \eqref{SU}, we can derive, for each $j= 0,1, 2$,
\begin{equation}
\label{SU1}
\sup_{\mathcal{B}\times T } \|S_{n,j}(u,2\widehat{\beta}(t))-S_j(u,2\beta(t)) \|\to 0,\quad as \quad n \to \infty.
\end{equation}
Analogous to the proof of $V_{1}(u,\bm{\widehat{\beta}})$, we can obtain
\begin{equation}
 \label{Vv21}
 \sup_{\mathcal{B}\times T}\| V_{21}(u,\bm{\widehat{\beta}})-v_{21}(u,\beta(u))\| \to 0,\quad as \quad n\to\infty,
 \end{equation}
 and 
 \begin{equation}
 \label{Vv22}
 \sup_{\mathcal{B}\times T}\| V_{21}^{2}(u,\bm{\widehat{\beta}})-v_{21}^{2}(u,\beta(u))\| \to 0,\quad as \quad n\to\infty.
 \end{equation}
Hence, from \eqref{Sig11}, \eqref{SU1}, \eqref{Vv21} and \eqref{Vv22}, the convergence of the first term can be demonstrated. 

For the second term of the right-hand side, let
\begin{equation}
\label{mu1}
\widehat{\mu}^{2}_0(u,\widehat{\beta}(u))-\mu^{2}_{0}(u)=(\widehat{\mu}_0(u,\widehat{\beta}(u))+\mu_{0}(u))(\widehat{\mu}_0(u,\widehat{\beta}(u))-\mu_{0}(u)),
\end{equation}
and
\begin{align}
\label{mu2}
\widehat{\mu}_0(u,\widehat{\beta}(u))-\mu_{0}(u)&=\{\widehat{\mu}_0(u,\widehat{\beta}(u))-\widehat{\mu}_0(u,\beta(u))\}+\{\widehat{\mu}_0(u,\beta(u))-\mu_{0}(u)\}\nonumber\\
&=H_{n}(u,\beta^{\ast})(\widehat{\beta}(u)-\beta(u))+n^{-1}\sum_{i=1}^{n}dM_{i}(u)/S^{\ast}_{n,0}(u,\beta(u)),
\end{align} 
where $H_{n}(u,\beta^{\ast})=-S_{n,1}^{\ast}(u,\beta^{\ast})\sum_{i=1}^{n}I(C_{i}\geq u)N_{i}(u)o_{i}(u)/nS_{n,0}^{\ast 2}(u,\beta^{\ast})$, and $\beta^{\ast}$ is between $\beta(u)$ and $\widehat{\beta}(u)$.\\
Under Conditions 1-5, by Lemma~\ref{lemA1} and \eqref{them1.1}, $H_{n}(u,\beta^{\ast})$ is bounded, and $\widehat{\beta}(u)$ uniform converges to $\beta(u)$. Hence, we can derive
\begin{equation}
\label{BU}
\sup_{u\in T}\| H_{n}(u,\beta^{\ast})(\widehat{\beta}(u)-\beta(u))\|\to 0,\quad as \quad n\to\infty.
\end{equation}
 For $n^{-1}\sum_{i=1}^{n}dM_{i}(u)/S^{\ast}_{n,0}(u,\beta(u))$ is empirical process, analogous to the proof of Theorem~\ref{theorem2}, using Theorem 8.3 (the uniform law of large numbers) of \citet{Pollard:1990}, we can obtain
 \begin{equation}
 \label{EU}
 \sup_{u\in T}\|n^{-1}\sum_{i=1}^{n}dM_{i}(u)/S^{\ast}_{n,0}(u,\beta(u))\| \to 0,\quad as \quad n\to\infty.
 \end{equation}
Under Conditions 1-5, $\widehat{\mu}_0(u,\widehat{\beta}(u))+\mu_{0}(u) $ is bounded, in conjunction with \eqref{mu1}, \eqref{mu2}, \eqref{BU} and \eqref{EU}, we obtain
\begin{equation}
\label{mu3}
\sup_{u\in T} \| \widehat{\mu}^{2}_0(u,\widehat{\beta}(u))-\mu^{2}_{0}(u)\| \to 0,\quad as \quad n\to\infty.
\end{equation}
Therefore, the second term converges to zero.

For the third term of the right-hand side, from \eqref{sig2}, obviously, converges to zero.\\
Hence,
\begin{equation}
\label{SS2}
\| \widehat{\Sigma}_2(t)-\Sigma_2(t)\|\to 0 ,\quad as \quad n\to\infty.
\end{equation}
Therefore, from \eqref{SS1} and \eqref{SS2}, we have $\widehat{\Sigma}(t)$ is consistent.
\end{proof}

\begin{proof}[of the asymptotic normality of $\widehat{\mu}_{0}(t,\widehat{\beta}(t))$]
Let
\begin{align}
\label{Mu1}
(nh)^{1/2}(\widehat{\mu}_0(t,\widehat{\beta}(t))-\mu_0(t)) & =(nh)^{1/2}(\widehat{\mu}_0(t,\widehat{\beta}(t))-\widehat{\mu}_0(t,\beta(t)))+(nh)^{1/2}(\widehat{\mu}_0(t,\beta(t))-\mu_0(t))\nonumber\\
& =H_{n}(t,\beta^{\ast})(nh)^{1/2}(\widehat{\beta}(t)-\beta(t))+(nh)^{1/2}\sum_{i=1}^{n}dM_i(t)/nS_{n,0}^{\ast}(t,\beta(t)). 
\end{align}
where
$H_{n}(t,\beta^{\ast})=-S_{n,1}^{\ast}(t,\beta^{\ast})\sum_{i=1}^{n}I(C_i\geq t)N_i(t)o_i(t)/nS_{n,0}^{\ast 2}(t,\beta^{\ast})$, and $\beta^{\ast}$ is between $\widehat{\beta}(t)$ and $\beta(t)$.

For the second term of right-hand side of \eqref{Mu1}, let
\begin{align}
\label{SM3}
&(nh)^{1/2}\sum_{i=1}^{n}dM_i(t)/nS_{n,0}^{\ast}(t,\beta(t))\nonumber\\
&=(nh)^{1/2}\sum_{i=1}^{n}dM_i(t)/nS_{0}^{\ast}(t,\beta(t))+n^{-1}(nh)^{1/2}\sum_{i=1}^{n}\{S_{n,0}^{\ast -1}(t,\beta(t))-S_{0}^{\ast -1}(t,\beta(t))\}dM_i(t).
\end{align}
Define 
\begin{equation}
\label{U3}
U_3(s)=n^{-1}(nh)^{1/2}\sum_{i=1}^{n}dM_i(s),
\end{equation}
Clearly, $E[U_3(s)]=0$.

Analogous to the proof of Lemma~\ref{lemB3}, using the functional central limit theorem of \citet{Pollard:1990}, we shall argument $U_3$ converges to Gaussian process $\xi_3$. Now we test the conditions (i)-(v) in \citet{Pollard:1990}.\\
 Under Conditions 1-5, by Lemma~\ref{lemA1}, we have, for any $s_1,s_2\in T$, 
\begin{align*}
&\lim_{n\to\infty}E(U_3(s_1)U_3(s_2))\\
&=\lim_{n\to\infty}E( n^{-1}h\sum_{i=1}^{n}(dM_i(s_1)dM_i(s_2)))\\
&=hE(\{I(C\geq s_1)\mu_0(s_1)\exp(\beta(s_1)z)o(s_1)\} \{I(C\geq s_2)\mu_0(s_2)\exp(\beta(s_2)z)o(s_2)\})+o_{p}(1)\\
&=\left \{ \begin{array}{ll} hE(p_{1}(s\mid z)p_{2}(s\mid z)\mu_0^{2}(s)\exp(2\beta(s)z))+o_{p}(1) & \qquad  s_1=s_2=s \\ 0 & \qquad s_1 \neq s_2 \end{array}\right.
\end{align*}
Then, by the classical multivariate central limit theorem for independent random vectors, the finite-dimensional distributions of $U_3$ converge to those of gaussian process $\xi_3$, which converge to zero, as $h$ gets to zero. Thus condition (ii) holds. Next, checking the tightness. Under Conditions 1-5, we know $\{ I(C_i\geq t)N_i(t)o_i(t), t \in T\}$ has finite points, and $\exp(\beta(t)z_i)o_i(t)$ are bounded variation functions. Thus, $\{ dM_i(t), t \in T \}$ is manageable, and the envelops can be chosen as constant $\bar{B}/\surd{n}$, then (i)(iii)(iv) holds.To verify (v),  for any $s_1,s_2 \in T$, define 
\[
\rho_n(s_1,s_2)=E(U_3(s_1)-U_3(s_2))^{2}, \quad \rho(s_1,s_2)=E(\xi_3(s_1)-\xi_3(s_2))^{2}.
\]
Further,
\begin{align*}
&\rho_n(s_1,s_2) = E(n^{-1}\sqrt{nh}\sum_{i=1}^{n}(dM_i(s_1)-dM_i(s_2)))^{2}\\
&=n^{-1}h\sum_{i=1}^{n}E( I(C_i\geq s_1)\mu^{2}_0(s_1)\exp(2\beta(s_1)z_i)o_i(s_1)+I(C_i\geq s_2)\mu^{2}_0(s_2)\exp(2\beta(s_2)z_i)o_i(s_2)).
\end{align*}
Clearly, $\{\rho_n\}$ is equicontinuous on $T$, and $\lim_{n\to\infty}\rho_n(s_1,s_2)=\rho(s_1,s_2),\rho$ is pseudometric on $T$.
Thus $\rho_n$ converges, uniformly on $T$, to $\rho$. And, let $\{s_1^{(n)}\},\{s_2^{(n)}\}$ be any two sequence in $T$, it follows that if $\rho(s_1^{(n)},s_2^{(n)})\to 0$, then $\rho_n(s_1^{(n)},s_2^{(n)})\to 0$, then (v) holds. Therefore, $U_3$ converges in distribution to Gaussian process on $T$, and covariance matrix is diagonal matrix, and the matrix is zero, when $h$ gets to zero. That is, $U_3$ converges in distribution zero, as $n\to \infty$, $h\to 0$ and $nh\to \infty$. Moreover, using the Strong Representation Theorem of \citet{Pollard:1990}, we have a new probability space and
\begin{equation}
\label{U-0}
\sup_{s\in T}\|U_3(s)-0\|\to 0,\quad  as \quad n\to \infty.
\end{equation}
By Lemma~\ref{lemA1} and Conditions 1-5, we can obtain
 \begin{equation}
 \label{SS3}
 \sup_{s\in T}\|S_{n,0}^{\ast -1}(t,\beta(t))-S_{0}^{\ast -1}(t,\beta(t))\|\to 0,\quad as \quad n\to\infty.
 \end{equation}
Then, by Lemma~\ref{lemA2} combined with \eqref{U-0} and \eqref{SS3}, we can derive, in probability,
\begin{equation}
\label{SM1}
n^{-1}(nh)^{1/2}\sum_{i=1}^{n}\{S_{n,0}^{\ast -1}(t,\beta(t))-S_{0}^{\ast -1}(t,\beta(t))\}dM_i(t)\to 0,\quad as \quad n\to\infty.
\end{equation}
which holds in the original probability space. And in analogy with the prove of $U_3(s)$, we can check that, in distribution,
\begin{equation}
\label{MSta}
(nh)^{1/2}\sum_{i=1}^{n}dM_i(t)/nS_{0}^{\ast}(t,\beta(t)) \to 0, \quad as \quad n\to\infty.
\end{equation}
Therefore, from \eqref{SM3}, \eqref{SM1} and \eqref{MSta}, we can obtain, in probability,
\begin{equation}
\label{SM2}
(nh)^{1/2}\sum_{i=1}^{n}dM_i(t)/nS_{n,0}^{\ast}(t,\beta(t))\to 0,\quad as \quad n\to\infty.
\end{equation}

For the first term of right-hand side of \eqref{Mu1}. Under Conditions 1-5 and Lemma~\ref{lemA1}, form \eqref{them1.1}, we obtain
\begin{equation}
\label{H1}
H_{n}(t,\beta^{\ast})\to -q_{1}(t)/q_{0}(t), \quad n\to \infty.
\end{equation}
By the assumption of $nh^{5}=o(1)$, from \eqref{cor1}, we can derive
\begin{equation}
\label{BB}
(nh)^{1/2}(\widehat{\beta}(t)-\beta(t))\to N(0, \nu_{0}\sigma_{1}^{-2}(t)\sigma_{2}(t)),\quad as \quad n\to\infty.
\end{equation}
Therefore, from \eqref{H1} and \eqref{BB}, we have 
\begin{equation}
\label{HB}
H_{n}(t,\beta^{\ast})(nh)^{1/2}(\widehat{\beta}(t)-\beta(t))\to N(\nu_{0}q_{0}^{-2}(t)q_{1}(t)\sigma_{1}^{-2}(t)\sigma_{2}(t)),\quad as \quad n\to\infty.
\end{equation}

Hence, form \eqref{Mu1}, \eqref{SM2} and \eqref{HB}, using Slutsky's theorem, we can obtain
\begin{equation}
\label{Mu2}
(nh)^{1/2}(\widehat{\mu}_{0}(t,\widehat{\beta}(t))-\mu_{0}(t))\to N(0,\Sigma_3(t)),\quad n\to\infty,
\end{equation}
where $\Sigma_3(t)=\nu_{0}q_{0}^{-2}(t)q_{1}(t)\sigma_{1}^{-2}(t)\sigma_{2}(t)$.
\end{proof}

\vspace*{-10pt}

\appendixthree
\section*{Appendix 3}

Here, we will show the simulation results about the local kernel estimators $\widehat{\beta}(t)$ with corresponding setting that $\beta(t)=\surd{t}$ and $\beta(t)$=0$\cdot$5(Beta($x/12,3,3$)+Beta($x/12,4,4$)), respectively, under sample sizes equal to 500. Those figures are displayed in the end of this paper.

\vspace*{-10pt}

\bibliographystyle{biometrika}
\bibliography{Reference}



\begin{figure}
\centering
\begin{minipage}[t]{0.425\linewidth}
\centering
\figurebox{13pc}{\linewidth}{}[be0.3.eps]
\end{minipage}
\begin{minipage}[t]{0.425\textwidth}
\centering
\figurebox{13pc}{\textwidth}{}[be0.5.eps]
\end{minipage}
\begin{minipage}[t]{0.425\textwidth}
\centering
\figurebox{13pc}{\textwidth}{}[se0.3.eps]
\end{minipage}
\begin{minipage}[t]{0.425\textwidth}
\centering
\figurebox{13pc}{\textwidth}{}[se0.5.eps]
\end{minipage}
\begin{minipage}[t]{0.425\textwidth}
\centering
\figurebox{13pc}{\textwidth}{}[cp0.3.eps]
\end{minipage}
\begin{minipage}[t]{0.425\textwidth}
\centering
\figurebox{13pc}{\textwidth}{}[cp0.5.eps]
\end{minipage}
\caption{(c1) and (c2): The true and the average of the local kernel estimator with bandwidth 0$\cdot$3 and 0$\cdot$5, respectively. (c3) and (c4): Comparison of empirical standard errors (ESE) and the estimated standard errors (MSE) for $\widehat{\beta}(t)$ with bandwidth 0$\cdot$3 and 0$\cdot$5, respectively; (c5) and (c6): Empirical coverage probabilities of the 95$\%$ confidence intervals for $\widehat{\beta}(t)$ with bandwidth 0$\cdot$3 and 0$\cdot$5, respectively.}
\label{fig4}
\end{figure}

\begin{figure}
\centering
\begin{minipage}[t]{0.425\linewidth}
\centering
\figurebox{13pc}{\linewidth}{}[be3.eps]
\end{minipage}
\begin{minipage}[t]{0.425\textwidth}
\centering
\figurebox{13pc}{\textwidth}{}[be5.eps]
\end{minipage}
\begin{minipage}[t]{0.425\textwidth}
\centering
\figurebox{13pc}{\textwidth}{}[se3.eps]
\end{minipage}
\begin{minipage}[t]{0.425\textwidth}
\centering
\figurebox{13pc}{\textwidth}{}[se5.eps]
\end{minipage}
\begin{minipage}[t]{0.425\textwidth}
\centering
\figurebox{13pc}{\textwidth}{}[cp3.eps]
\end{minipage}
\begin{minipage}[t]{0.425\textwidth}
\centering
\figurebox{13pc}{\textwidth}{}[cp5.eps]
\end{minipage}
\caption{(d1) and (d2): The true and the average of the local kernel estimator with bandwidth 0$\cdot$3 and 0$\cdot$5, respectively. (c3) and (c4): Comparison of empirical standard errors (ESE) and the estimated standard errors (MSE) of $\widehat{\beta}(t)$ with bandwidth 0$\cdot$3 and 0$\cdot$5, respectively; (d5) and (d6): Empirical coverage probabilities of the 95$\%$ confidence intervals for $\widehat{\beta}(t)$ with bandwidth 0$\cdot$3 and 0$\cdot$5, respectively.}
\label{fig5}
\end{figure}

\end{document}